\documentclass[11pt,a4paper, reqno]{amsart}
\usepackage{mathrsfs}

\usepackage{amsmath,amsfonts,verbatim}
\usepackage{latexsym}
\usepackage{amssymb,leftidx}
\usepackage{extarrows}
\usepackage{overpic}
\usepackage{color}
\usepackage{epsfig}
\usepackage{subfigure}
\usepackage{tikz}
\usepackage{bbm}
\usepackage{tikz, caption}
\usepackage[font=small,labelfont=bf]{caption}
\usepackage{mathtools}

\usepackage[colorlinks=true]{hyperref}
\hypersetup{urlcolor=blue, citecolor=blue}
\usepackage{graphicx, amssymb, amsmath, amsthm}
\numberwithin{equation}{section}
\usepackage{color}
\usepackage{amscd}
\usepackage{tikz}
\def\pa{\partial}

\let\Im=\undefined\DeclareMathOperator*{\Im}{Im}

\newcommand{\R}{\mathbb{R}}
\newcommand{\C}{\mathbb{C}}

\newcommand{\Z}{\mathbb{Z}}

\newcommand{\N}{\mathbb{N}}

\newcommand{\Sa}{\mathbb{S}_\sigma^1}

\newtheorem{theorem}{Theorem}[section]

\newtheorem{lemma}[theorem]{Lemma}

\newtheorem{conjecture}[theorem]{Conjecture}
\newtheorem{proposition}[theorem]{Proposition}

\theoremstyle{definition}

\newtheorem{remark}[theorem]{Remark}

\makeatletter
\newcommand{\Extend}[5]{\ext@arrow0099{\arrowfill@#1#2#3}{#4}{#5}}
\makeatother

\begin{document}
\title[Bochner-Riesz means]{Bochner-Riesz means \\ on a conical singular manifold}

\author{Qiuye Jia}
\address{Department of Mathematics, the Australian National University; }
\email{Qiuye.Jia@anu.edu.au; }

\author{Junyong Zhang}
\address{Junyong Zhang
\newline \indent Department of Mathematics, Beijing Institute of Technology, Beijing 100081}
\email{zhang\_junyong@bit.edu.cn}

\author{Jiiqang Zheng}
\address{Jiqiang Zheng
\newline \indent Institute of Applied Physics and Computational Mathematics,
Beijing, 100088, China.
\newline\indent
National Key Laboratory of Computational Physics, Beijing 100088, China}
\email{zheng\_jiqiang@iapcm.ac.cn, zhengjiqiang@gmail.com}

\begin{abstract}
We prove a sharp $L^p$-boundedness criterion for Bochner-Riesz multipliers on flat cones $X = (0,\infty) \times \mathbb{S}_\sigma^1$. The operator $S_\lambda^\delta(\Delta_X)$ is bounded on $L^p(X)$ for $1 \leq p \leq \infty$, $p \neq 2$, if and only if $\delta > \delta_c(p,2) = \max\left\{ 0, 2\left| 1/2 - 1/p \right| - 1/2 \right\}$. This result is also applicable to the infinite sector domain with Dirichlet or Neumann boundary, resolving the critical exponent problem in this wedge setting.
\end{abstract}

 \maketitle

\begin{center}
 \begin{minipage}{100mm}
 \end{minipage}
 \end{center}

% \tableofcontents %%Žú±íÔËÐÐÄ¿ÂŒ£¬Èç¹û²»ÒªÄ¿ÂŒ£¬¿ÉÒÔ°ÑÕâŸä»°×¢ÊÍµô¡£

\section{Introduction}

\noindent

We investigate the $L^p$-boundedness of the Bochner-Riesz multiplier on flat cones. 
In the Euclidean space $\R^n$, for $\delta\geq0$ and $\lambda>0$, the Bochner-Riesz operator  $S_{\lambda}^\delta(\Delta_{\R^n})$ of order $\delta$ is defined by
\begin{equation}
S_{\lambda}^\delta(\Delta_{\R^n}) f(x)=\int_{|\xi|\leq\lambda} e^{ix\cdot\xi}\big(1-\frac{|\xi|^2}{\lambda^2}\big)^\delta \hat{f}(\xi)\, d\xi
\end{equation}
where the Laplacian \footnote{Throughout this paper, our sign convention for Laplacians is that $\Delta_{\bullet}$ is a positive operator. }  $\Delta_{\R^n}=-\sum_{i=1}^n \partial^2_{x_i}$ and $\hat{f}$ is the Fourier transform of $f$.
The problem known as the Bochner-Riesz problem is to determine the optimal index $\delta$ for $1\leq p\leq +\infty$ such that 
$S_{\lambda}^\delta(\Delta_{\R^n})f$ converges to $f$ in $L^p({\mathbb R^2})$ (as $\lambda\to +\infty$) for every $f\in L^p({\mathbb R^2})$. When $p = 2$, the convergence holds true if and only if $\delta\geq 0$ by
Plancherel’s theorem. The Bochner-Riesz conjecture is as follows

\begin{conjecture}[Bochner-Riesz conjecture]\label{conj:BR}
For $p\neq2$, the Bochner-Riesz mean $S_R^\delta(\Delta_{\R^n})f(x)$ converges in $L^p(\R^n)$ if and only if
$\delta>\delta_c(p,n)$, where
\begin{equation}\label{equ:BRconj}
\delta_c(p,n):=\max\left\{0,n\big|\tfrac12-\tfrac1p\big|-\tfrac12\right\},\quad \forall\;p\in[1,+\infty].
\end{equation}

\end{conjecture}

In two dimensions, the conjecture has been proved by L. Carleson and P. Sj\"olin \cite{CS} by means of a more general theorem on oscillatory integral operators.  For a variant of their proof, we refer to L. H\"ormander\cite{Hor},  C. Fefferman \cite{Fef2}  and A. Cordoba \cite{Cor}. In higher dimensions, only partial results are known, see e.g. \cite{Bour, BG, GOWWZ, GWX, Lee, Stein, Ta3}  and references therein.\vspace{0.1cm}

The Bochner-Riesz means for  Schr\"odinger operator with potentials and the Laplacian on manifolds have attracted a lot of attention and have been studied extensively by many authors, see 
 \cite{COSY, CSo, DOS, GHS, KST, H2, JLR1, JLR, LR, Sogge87, T2} and reference therein. In particular, we refer to Sogge \cite{Sogge87}, Christ-Sogge \cite{CSo} for the Bochner-Riesz means on compact manifolds and
 Guillarmou-Hassell-Sikora \cite{GHS} for the Bochner-Riesz means on the non-trapping asymptotically conic manifold. Recently, 
Lee and Ryu \cite{LR} studied the Bochner-Riesz means problem for Schr\"odinger operator with Hermite operator, and later Jeong, Lee and Ryu \cite{JLR} for the twisted Laplacian in $\R^2$. Recently, %Miao-Yan-Zhang 
Miao, Yan and the second author \cite{MYZ} studied the Bochner-Riesz problem for the Schr\"odinger operators with a scaling critical magnetic potential. 
However, the picture of the Bochner-Riesz problem in general geometric setting or of variable coefficient operators
%associated with general elliptic operators, in particular the variable coefficient operators, 
is far from complete even in two dimension, see Sogge \cite{Sogge87, Sogge02}.  Motivated by this, we 
study the Bochner-Riesz problem on a flat cone, which is a conical singular manifold and potentially study the Bochner-Riesz problem on polygon in future. \vspace{0.2cm}

Let $(X,g)$ be a flat cone, where $X=C(\mathbb{S}_\sigma^1)=(0,\infty)\times\mathbb{S}_\sigma^1$ is a product cone over the circle, $\mathbb{S}_\sigma^1=\R/2\pi\sigma\Z$ with radius $\sigma>0$ and the metric $g=dr^2+r^2d\theta^2$. This space is a generalization of the Euclidean space $\R^2$ since $X=\R^2$ when $\sigma=1$ after adding the cone tip at $r=0$. 
The geometric difference with $\R^2$ leads to new phenomena in analysis. 
%The slightly geometry difference between $\R^2$ leads to different phenomena in the analysis problems. 
For example, the diffraction phenomenon of the wave propagator on cones was studied by \cite{CT, CT1, FHH}.
The dispersive and Strichartz estimates for the Schr\"odinger and wave equations on $C(\Sa)$ were studied in \cite{Ford, BFM}. 
Another interesting aspect of $C(\Sa)$
%The more interesting 
is that the flat cone can be regarded as local model of 
the polygonal domains $\Omega\subset \R^2$ or more generally Euclidean surfaces with conical singularities, see \cite{BFM, BFM18, HHM}.

The purpose of this paper is to investigate the Bochner-Riesz operator  $S_{\lambda}^\delta(\Delta_X): = (1- {\Delta_X}/{\lambda^2} )^\delta_{+}$ of order $\delta$ associated with the positive Laplace-Beltrami operator $\Delta_X$ in $X$, which is defined by, for $\delta\geq 0$ and $\lambda>0$,
\begin{eqnarray}\label{oper:BR}
S_{\lambda}^\delta(\Delta_X)
%&=&\Big(1-\frac{\LL_{\A}}{\lambda^2}\Big)^\delta_{+}\nonumber\\
%&=&
=\int_0^{\lambda} \Big(1-\frac{\rho^2}{\lambda^2}\Big)^\delta  dE_{\sqrt{\Delta_X}} (\rho), 
\end{eqnarray}
where $E_{\sqrt{\Delta_X}} (\rho) $ is a spectral resolution of $\sqrt{\Delta_X}$. \vspace{0.2cm}

Now, we state our main result.

\begin{theorem}\label{thm:main}
Let $1\leq p\leq+\infty$ and $\delta>\delta_c(p,2)$ be given by \eqref{equ:BRconj}, which means
\begin{equation}
\delta_c(p,2)=\max\left\{0,2\big|\tfrac12-\tfrac1p\big|-\tfrac12\right\}=\begin{cases}
\frac12-\frac2p,\quad &\text{if}\quad p\geq4\\
0,\quad &\text{if}\quad \frac43<p<4\\
\frac2p-\frac32,\quad &\text{if}\quad p\leq\frac43.
\end{cases}
\end{equation}
Then, there holds
\begin{equation}\label{est:BRcone}
\big\|S_\lambda^\delta(\Delta_X)f\big\|_{L^p(X)}\leq C\|f\|_{L^p(X)},
\end{equation}
where the constant $C$ is independent of $\lambda>0$.
\end{theorem}

\begin{remark}
%Since Conjecture~\ref{conj:BR}
The condition $\delta>\delta_c(p,2)$ is necessary by case $X = \R^n$ when $p\neq2$, as stated in Conjecture~\ref{conj:BR} for $n=2$ and reference after it.
\end{remark}

The proof of Theorem \ref{thm:main} is based on the Schwartz kernel of the spectral measure $dE_{\sqrt{\Delta_X}}(\rho;x,y)$ constructed by the second author \cite{Zhang}
and the oscillatory integral theory of Stein \cite{Stein} and H\"ormander \cite{Hor}. However, we have to overcome difficulties caused by the singularity of the vertex and the diffraction effect. 
%\rv{Now let us figure out some comment interesting in general.}\vspace{0.2cm}
Next we discuss some settings that our main result applies to and make comparison with some closely related problems.
\begin{itemize}
\item With minor modifications, Theorem \ref{thm:main} applies to the two-dimensional infinite sector domain $S_{\alpha}$ with angle $\alpha\in (0,2\pi)$ and the Laplacian on $S_{\alpha}$ with Dirichlet (resp. Neumann) boundary conditions.
To see this, let us recall the procedure outlined in \cite[Section 2]{HHM}. Begin with two copies $S_{\alpha}$ and $\kappa S_{\alpha}$ of the  infinite sector domains,
where $\kappa$ is a reflection of the plane, a cone $X=C(\mathbb{S}^1_{\sigma})$ with $\sigma=2 \times \frac{\alpha}{2\pi}=\frac{\alpha}{\pi}$ is then obtained by taking the
formal union $S_{\alpha}\cup \kappa S_{\alpha}$, where two corresponding infinity sides are identified pointwise.
To derive estimates for functions on $S_{\alpha}$ with Dirichlet (resp. Neumann) boundary conditions, we apply Theorem \ref{thm:main} to their even (resp. odd) extensions, which are functions on $X$.
%Then the corresponding function on $X$ is obtained via 
\vspace{0.1cm}

\item 
We plan to address the Bochner-Riesz problem on polygon domains in the future.
The doubling procedure above can also be applied to a polygonal domain (see also \cite[Section 2]{HHM} for more details), which is locally isometric to our cone $X=C(\mathbb{S}^1_\sigma)$ near each vertex. 
%Our current considering cones serve as models for our eventual study of the Bochner-Riesz problems on polygonal domains.
However, the Bochner-Riesz problems about the $L^p$-bounds need to understand the long time dynamics of the geodesic flow (billiard flow in this case), which is complicated due to the polygon's boundary reflection and diffraction at corners. To our best knowledge, there is no sharp result about the Bochner-Riesz problems on manifold with or without boundary, except the special torus (e.g. see Fefferman \cite{Fef1, Fef2} and \cite[Theorem 3.8]{SW}). We refer to Sogge \cite{Sogge87, Sogge02} for partial results on manifolds with or without boundary.  \vspace{0.1cm}

\item To study the Bochner-Riesz problems,  one needs %diffractive geometry 
to study the diffractive effect of conical singular operator in Ford-Wunsch \cite{FW}. 
 In contrast to Ford-Wunsch \cite{FW}, in which diffraction  and the estimates in weighted $L^2$ space are studied, the $L^p$-framework estimates on Bochner-Riesz means need
to elaborate properties of the kernel which captures  the decay and the oscillation behavior.  \vspace{0.1cm}

\end{itemize}

This paper is organized as follows. In Section \ref{Sec:spect}, we construct the  Bochner-Riesz kernel \eqref{oper:BR}
based on the spectral measure of the operator $\sqrt{\Delta_X}$.
%the measure kernel of the operator $\Delta_X$, which is our starting point.   
In Section \ref{Sec:mainthm}, we prove Theorem \ref{thm:main} 
by assuming Proposition \ref{prop:keyTD1+k}, whose proof occupies the whole
Section \ref{sec:TD1+kterm}. In Appendix~\ref{Sec:appendix}, we prove some auxiliary lemmas for estimating oscillatory integrals throughout the paper.\vspace{0.2cm}

{\bf Acknowledgments:} \quad  The authors would like to thank Andrew Hassell and Xiaoqi Huang for their helpful discussions. 
J. Zhang was supported by National Natural Science Foundation of China(12171031, 12531005) and Beijing Natural Science Foundation(1242011);
Q. Jia was supported by the Australian Research Council through grant FL220100072; J. Zheng was supported by National Key R\&D program of China: 2021YFA1002500 and NSFC Grant 12271051.

%%%%%%%%%%%%%%%%%%%%%%%%%%%%%%%%%%%%%%%%%%%%%%%%%%%%%%%%%%%%%%%%%%%%%%%%%%%%%%%%%%%%%%%%%%%%%%%%%%%%%%%%%%%%%%%%%%%%%%%%%%%%%%%%%%%%%%%%%%%%%%

%%%%%%%%%%%%%%%%%%%%%%%%%%%%%%%%%%%                                %%%%%%%%%%%%%%%%%%%%%%%%%%%%%%%%%%%%%%%%%%%%%%%%%%%%%%%%%%%%%%%%%%%%%%

%%%%%%%%%%%%%%%%%%%%%%%%%%%%%%%%%%%%%%%%%%%%%%%%%%%%%%%%%%%%%%%%%%%%%%%%%%%%%%%%%%%%%%%%%%%%%%%%%%%%%%%%%%%%%%%%%%%%%%%%%%%%%%%%%%%%%%%%%%%%%%

\section{The spectral measure and Bochner-Riesz mean}\label{Sec:spect}

We recall the representation of the  spectral measure $dE_{\sqrt{\Delta_X}}(\rho;x,y)$ constructed in \cite{Zhang} and then provide the 
kernel of the Bochner-Riesz mean, which captures both the decay and oscillation properties.

\subsection{The kernel of resolvent and spectral measure}

\begin{proposition}[Spectral measure kernel]\label{prop:spect}

 Let $X=C(\mathbb{S}_\sigma^1)$ be a product cone over the circle, $\mathbb{S}_\sigma^1=\R/2\pi\sigma\Z$ with radius $\sigma>0$, equipped with the metric $g=dr^2+r^2d\theta^2$
 and let $\Delta_X$ be the Laplace-Beltrami operator on $X$. Let $x=(r_1,\theta_1),\;y=(r_2,\theta_2)$ in $X=C(\mathbb{S}_\sigma^1)$ and denote
 \begin{align}\label{equ:mjdef}
\mathbf{m}_j=&\left(r_1-r_2,\sqrt{2\big[1-\cos(\theta_1-\theta_2+2j\sigma\pi)\big]r_1r_2}\right)\in\R^2,\quad j\in\Z, \\ \label{equ:ndef}
\mathbf{n}_s=&(n_1,n_2)=\big(r_1+r_2,\sqrt{2(\cosh s-1)r_1r_2}\big)\in\R^2.
\end{align}
Then the Schwartz kernel of the spectral measure $dE_{\sqrt{\Delta_X}}(\rho;x,y)$ can be written as 
\begin{equation}\label{equ:spect}
\begin{split}
dE_{\sqrt{\Delta_X}}(\rho;x,y)=&\frac1{4\pi^2}\frac{\rho}{\pi i}\sum_{\{j\in\Z: 0\leq|\theta_1-\theta_2+2j\sigma \pi|\leq \pi\}} \mathrm{Spect}(\rho, {\bf m}_j)\\
&-\frac1{4\pi^2\sigma}\frac{\rho}{\pi^2 i}\int_0^\infty \mathrm{Spect}(\rho, {\bf n}_s)
A_\sigma(s,\theta_1,\theta_2)\;ds,
\end{split}
\end{equation}
where
\begin{align}\label{equ:Asigmadef}
A_\sigma(s,\theta_1,\theta_2)=&{\rm Im}\left(\frac{e^{i\frac1\sigma(\pi-(\theta_1-\theta_2))}}{e^\frac{s}{\sigma}-e^{i\frac1{\sigma}(\pi-(\theta_1-\theta_2))}}
-\frac{e^{-i\frac1\sigma(\pi+(\theta_1-\theta_2))}}{e^\frac{s}{\sigma}-e^{-i\frac1{\sigma}(\pi+(\theta_1-\theta_2))}}\right)\\\nonumber
=&\frac12\frac{\sin\big(\tfrac{\pi-(\theta_1-\theta_2)}{\sigma}\big)}{\cosh\frac{s}{\sigma}-\cos\big(\tfrac{\pi-(\theta_1-\theta_2)}{\sigma}\big)}
+\frac12\frac{\sin\big(\tfrac{\pi+(\theta_1-\theta_2)}{\sigma}\big)}{\cosh\frac{s}{\sigma}-\cos\big(\tfrac{\pi+(\theta_1-\theta_2)}{\sigma}\big)},
\end{align}
and for ${\bf m}\in\R^2$
\begin{equation}\label{def:spec}
\mathrm{Spect}(\rho, {\bf m})=\int_{\R^2} e^{-i\mathbf{m}\cdot\xi}\left(\frac{1}{|\xi|^2-(\rho^2+i0)}-\frac{1}{|\xi|^2-(\rho^2-i0)}\right)\;d\xi.
\end{equation}

\end{proposition}

\begin{remark} Notice that
\begin{align}\label{equ:djdef}
|\mathbf{m}_j|=d_j(r_1,r_2,\theta_1,\theta_2)=&\sqrt{r_1^2+r_2^2-2\cos(\theta_1-\theta_2+2j\sigma\pi)r_1r_2},\\\label{equ:dsdef}
|\mathbf{n}_s|=d_s(r_1,r_2,\theta_1,\theta_2)=&\sqrt{r_1^2+r_2^2+2r_1r_2\cosh(s)},\quad s\in[0,+\infty).
\end{align}
In particular $\sigma=1$,  $A_\sigma(s,\theta_1,\theta_2)$ vanishes and $j=0$, thus $|\mathbf{m}_0|=|x-y|$ and
\begin{align*}
dE_{\sqrt{\Delta_X}}(\rho;x,y)=dE_{\sqrt{-\Delta_{\R^2}}}(\rho; x, y).
\end{align*}
\end{remark}

\begin{remark} From \eqref{equ:spect}, we have
\begin{equation}\label{equ:spect-scaling}
\begin{split}
dE_{\sqrt{\Delta_X}}(\lambda\rho;x,y)=dE_{\sqrt{\Delta_X}}(\rho; \lambda x, \lambda y),
\end{split}
\end{equation}
which essentially is due to the dilation invariant of the Laplace-Beltrami operator $\Delta_X$.
\end{remark}

\begin{proof} 
The proof is modified from the argument in \cite[Section 2]{Zhang}. We sketch the proof here for readers' convenience.
We write the positive Laplacian on $X$ 
\begin{equation*}
\Delta_X=-\partial_r^2-\frac1r\partial_r+\frac1{r^2}\Delta_{\mathbb{S}_\sigma^1}
\end{equation*}
where $\Delta_{\mathbb{S}_\sigma^1}=-\partial_\theta^2$ is the Laplacian operator on $\mathbb{S}_\sigma^1$. 
Let $\nu_k=|k|/\sigma$ and 
\begin{equation}\label{eig-f}
\varphi_k(\theta)=\frac1{\sqrt{2\pi\sigma}}e^{-\frac{ik\theta}{\sigma}},\qquad k\in\Z,
\end{equation}
then $\nu_k$ and $\varphi_k$ are eigenvalues and eigenfunctions of operator $\Delta_{\mathbb{S}_\sigma^1}$ such that
\begin{equation}
-\partial_\theta^2\varphi_k(\theta)=\nu_k^2\varphi_k(\theta).
\end{equation}
By the functional calculus of Cheeger-Taylor 
%separation of variables 
 (see e.g. \cite[(8.47)]{Taylor}),  we obtain the kernel of the operator  $e^{-it\Delta_X}$
\begin{equation}\label{equ:ker-S}
e^{-it\Delta_X}(r_1,\theta_1,r_2,\theta_2)=\sum_{k\in\Z}\varphi_{k}(\theta_1)\overline{\varphi_{k}(\theta_2)}K_{\nu_k}(t,r_1,r_2)
\end{equation}
where $K_{\nu}(t,r_1,r_2)$ is given by
\begin{equation}\label{equ:knukdef}
\begin{split}
  K_{\nu}(t,r_1,r_2)&=\int_0^\infty e^{-it\rho^2}J_{\nu}(r_1\rho)J_{\nu}(r_2\rho) \,\rho d\rho
  \\&=\lim_{\epsilon\searrow 0} \int_0^\infty e^{-(\epsilon+it)\rho^2}J_{\nu}(r_1\rho)J_{\nu}(r_2\rho) \,\rho d\rho\\
  =\frac{e^{-\frac{r_1^2+r_2^2}{4it}}}{2it}&
  \Big(\frac1{\pi}\int_0^\pi e^{\frac{r_1r_2}{2it}\cos(s)} \cos(\nu s) ds-\frac{\sin(\nu\pi)}{\pi}\int_0^\infty e^{-\frac{r_1r_2}{2it}\cosh s} e^{-s\nu} ds\Big),
  \end{split}
\end{equation}
where in the last equality we use the Weber's second exponential integral \cite[Section 13.31 (1)]{Watson} and the integral representation of the modified Bessel function
\begin{equation*}
I_\nu(z)=\frac1{\pi}\int_0^\pi e^{z\cos(s)} \cos(\nu s) ds-\frac{\sin(\nu\pi)}{\pi}\int_0^\infty e^{-z\cosh s} e^{-s\nu} ds.
\end{equation*}
Recalling $\nu_k=|k|/\sigma$ and \eqref{eig-f}, we obtain
\begin{equation}
\begin{split}
& \sum_{k\in\Z}\varphi_{k}(\theta_1)\overline{\varphi_{k}(\theta_2)}\cos(\nu_k s)=
\sum_{k\in\Z}\frac{1}{2\pi\sigma}e^{-i\frac{k}{\sigma}(\theta_1-\theta_2)}\frac{e^{i\frac{ks}{\sigma}}+e^{-i\frac{ks}{\sigma}}}2\\
&= \frac12\sum_{j\in\Z}\big[\delta(\theta_1-\theta_2+s+2j\pi \sigma)
 +\delta(\theta_1-\theta_2-s +2j\pi\sigma)\big],
\end{split}
\end{equation} 
where we used the Poisson's formula
\begin{equation}
\sum_{j\in\Z} \delta(x-Tj)=\sum_{k\in\Z} \frac1T e^{i 2\pi \frac{k}{T}x}, \quad T=2\pi \sigma.
\end{equation}
Thus, on the one hand, we have 
\begin{equation}\label{eq:S1}
\begin{split}
&\frac1{\pi}\int_0^\pi e^{\frac{r_1r_2}{2it}\cos(s)}  \sum_{k\in\Z}\varphi_{k}(\theta_1)\overline{\varphi_{k}(\theta_2)}\cos(\nu_k s) ds\\ 
 =&\frac1{2\pi}\sum_{\{j\in\Z: 0\leq |\theta_1-\theta_2+2j\sigma\pi|\leq \pi\}} e^{\frac{r_1r_2}{2it}\cos(\theta_1-\theta_2+2j\sigma\pi)}.
 \end{split}
\end{equation}
On the other hand, by using \begin{equation*}
\sum_{k=1}^\infty e^{ikz}=\frac{e^{iz}}{1-e^{iz}},\qquad \mathrm{Im} z>0,
\end{equation*}
we have
\begin{equation}\label{eq:S2}
\begin{split}
&\sum_{k\in\Z}\sin(\pi\frac{|k|}\sigma)e^{-\frac{|k|s}{\sigma}}e^{-i\frac k\sigma(\theta_1-\theta_2)}\\
=&\sum_{k\geq 1} \frac{e^{i\frac{k}{\sigma}\pi}-e^{-i\frac{k}{\sigma}\pi}}{2i} e^{-\frac{ks}{\sigma}}\big(e^{-i\frac k{\sigma}(\theta_1-\theta_2)}+e^{i\frac k{\sigma}(\theta_1-\theta_2)}\big)\\
=& \sum_{k\geq 1}\Im \big(e^{i\frac{k}\sigma(\pi-(\theta_1-\theta_2)+is)}-e^{i\frac{k}\sigma(-\pi-(\theta_1-\theta_2)+is)}\big)\\
=& \Im \Big(\frac{e^{i\frac1\sigma(\pi-(\theta_1-\theta_2))}}{e^{\frac s\sigma}-e^{i\frac1\sigma(\pi-(\theta_1-\theta_2))}}-\frac{e^{-i\frac1\sigma(\pi+(\theta_1-\theta_2))}}{e^{\frac s\sigma}-e^{-i\frac1\sigma(\pi+(\theta_1-\theta_2))}}\Big)\\
=&A_{\sigma}(s,\theta_1,\theta_2).
\end{split}
\end{equation}
Plugging \eqref{eq:S1} and \eqref{eq:S2} into \eqref{equ:ker-S}, we obtain
\begin{equation}\label{S-kernel} 
\begin{split}
e^{-it\Delta_X}(x,y)&=
\frac1{4\pi}\frac{e^{-\frac{r_1^2+r_2^2}{4it}} }{it}
\sum_{\{j\in\Z: 0\leq |\theta_1-\theta_2+2j\sigma\pi|\leq \pi\}} e^{\frac{r_1r_2}{2it}\cos(\theta_1-\theta_2+2j\sigma\pi)}\\&
-\frac{1}{2\pi^2\sigma}\frac{e^{-\frac{r_1^2+r_2^2}{4it}} }{it}
 \int_0^\infty e^{-\frac{r_1r_2}{2it}\cosh s} A_{\sigma}(s,\theta_1,\theta_2) ds.
\end{split}
\end{equation}
We first note that when $z\in \{z\in\C: \Im(z)>0\}$, we have
$$(s-z)^{-1}=\frac1{i}\int_0^\infty e^{-ist} e^{iz t}dt,\quad \forall s\in\R,$$
thus we obtain, for $z=\rho^2+i\epsilon$ with $\epsilon>0$, 
\begin{equation}\label{res+}
\begin{split}
(\Delta_X-(\rho^2+i0))^{-1}&=\frac1{i}\lim_{\epsilon\to 0^+}\int_0^\infty e^{-it\Delta_X} e^{it(\rho^2+i\epsilon)}dt.
\end{split}
\end{equation}
From \eqref{res+} and \eqref{S-kernel} , we obtain
\begin{equation*}
\begin{split}
&(\Delta_X-(\rho^2+i0))^{-1}\\&=
\frac1{4\pi i}
\sum_{\{j\in\Z: 0\leq |\theta_1-\theta_2+2j\sigma\pi|\leq \pi\}} \lim_{\epsilon\to 0^+}\int_0^\infty e^{it(\rho^2+i\epsilon)}\frac{e^{-\frac{r_1^2+r_2^2}{4it}} }{it} e^{\frac{r_1r_2}{2it}\cos(\theta_1-\theta_2+2j\sigma\pi)} \,dt\\&
-\frac{1}{2\pi^2\sigma i}
 \int_0^\infty \lim_{\epsilon\to 0^+}\int_0^\infty e^{it(\rho^2+i\epsilon)} \frac{e^{-\frac{r_1^2+r_2^2}{4it}} }{it} e^{-\frac{r_1r_2}{2it}\cosh s} \,dt A_{\sigma}(s,\theta_1,\theta_2) ds.
 \end{split}
\end{equation*}
To further compute the kernel of resolvent, we recall \cite[Lemma 2.3]{Zhang}.
\begin{lemma} Let  $z=\rho^2+i\epsilon$ with $\epsilon>0$. Then
\begin{equation}\label{equ:m}
\begin{split}
\int_0^\infty \frac{e^{-\frac{r_1^2+r_2^2}{4it}} }{it} e^{\frac{r_1r_2}{2it} \cos(\theta_1-\theta_2+2j\sigma\pi)} e^{iz t}dt=\frac{i}\pi\int_{\R^2} \frac{e^{-i{\bf m}_j\cdot {\xi}}}{|\xi|^2-z} \, d{ \xi},
\end{split}
\end{equation}
and
\begin{equation}\label{equ:n}
\begin{split}
\int_0^\infty \frac{e^{-\frac{r_1^2+r_2^2}{4it}} }{it} e^{-\frac{r_1r_2}{2it}\cosh s} e^{iz t}dt=\frac i \pi\int_{\R^2}\frac{e^{-i{\bf n}_s\cdot {\xi}}}{|\xi|^2-z} \, d{ \xi},
\end{split}
\end{equation}
where $\xi=(\xi_1,\xi_2)\in\R^2$ and ${\bf m}_j, {\bf n}_s\in\R^2$ are given in \eqref{equ:mjdef}.

\end{lemma}

From \eqref{equ:m} and \eqref{equ:n}, it follows 
\begin{equation}\label{equ:res+}
\begin{split}
&(\Delta_X-(\rho^2+i0))^{-1}\\&=
\frac1{4\pi}
\sum_{\{j\in\Z: 0\leq |\theta_1-\theta_2+2j\sigma\pi|\leq \pi\}} \lim_{\epsilon\to 0^+}\frac1\pi\int_{\R^2} \frac{e^{-i{\bf m}_j\cdot {\xi}}}{|\xi|^2-(\rho^2+i\epsilon)} \, d{ \xi}\\&
-\frac{1}{2\pi^2\sigma }
 \int_0^\infty \lim_{\epsilon\to 0^+}\frac1\pi\int_{\R^2} \frac{e^{-i{\bf n}_s\cdot {\xi}}}{|\xi|^2-(\rho^2+i\epsilon)} \, d{ \xi} \,A_{\sigma}(s,\theta_1,\theta_2) ds.
 \end{split}
\end{equation}
The kernel of $(\Delta_X-(\rho^2-i0))^{-1}$ can be obtained via taking complex conjugation 
 \begin{equation}\label{out-inc}
\begin{split}
&(\Delta_X-(\rho^2-i0))^{-1}=\overline{(\Delta_X-(\rho^2+i0))^{-1}}
\\&=
\frac1{4\pi}
\sum_{\{j\in\Z: 0\leq |\theta_1-\theta_2+2j\sigma\pi|\leq \pi\}} \lim_{\epsilon\to 0^+}\frac1\pi\int_{\R^2} \frac{e^{-i{\bf m}_j\cdot {\xi}}}{|\xi|^2-(\rho^2-i\epsilon)} \, d{ \xi}\\&
-\frac{1}{2\pi^2\sigma }
 \int_0^\infty \lim_{\epsilon\to 0^+}\frac1\pi\int_{\R^2} \frac{e^{-i{\bf n}_s\cdot {\xi}}}{|\xi|^2-(\rho^2-i\epsilon)} \, d{ \xi} \,A_{\sigma}(s,\theta_1,\theta_2) ds.
 \end{split}
\end{equation}
According to Stone’s formula, the spectral measure is related to the resolvent 
 \begin{equation}\label{stone}
 dE_{\sqrt{\Delta_X}}(\rho)=\frac{d}{d\rho}dE_{\sqrt{\Delta_X}}(\rho)\,d\rho=\frac{\rho}{\pi i}\big(R(\rho+i0)-R(\rho-i0)\big)\, d\rho
 \end{equation}
 where
$$R(\rho\pm i0)=\lim_{\epsilon\searrow 0}(\Delta_X-(\rho^2\pm i\epsilon))^{-1}.$$
 From \eqref{stone}, \eqref{out-inc} and \eqref{equ:res+}, we have
 \begin{equation*}
 \begin{split}
 &dE_{\sqrt{\Delta_X}}(\rho;x,y)\\
& =
\frac1{4\pi} \frac{\rho}{\pi^2 i}
\sum_{\{j\in\Z: 0\leq |\theta_1-\theta_2+2j\sigma\pi|\leq \pi\}} \int_{\R^2} e^{-i{\bf m}_j\cdot {\xi}}\Big(\frac1{|\xi|^2-(\rho^2+i0)}-\frac1{|\xi|^2-(\rho^2-i0)}\Big) \, d{ \xi}
 \\&
-\frac{1}{2\pi^2 \sigma} \frac{\rho}{\pi^2 i}
 \int_0^\infty \int_{\R^2} e^{-i{\bf n}_s\cdot {\xi}}\Big(\frac1{|\xi|^2-(\rho^2+i0)}-\frac1{|\xi|^2-(\rho^2-i0)}\Big) \, d{ \xi}
A_{\sigma}(s,\theta_1,\theta_2) ds.
\end{split}
\end{equation*}
Therefore, recalling \eqref{def:spec}, we have proved \eqref{equ:spect}. 
%as desired.
\end{proof}

\subsection{The kernel of Bochner-Riesz means}\label{Subsec:BRmcone}

In this subsection, we will provide the kernel of the Bochner-Riesz means of order $\delta$ in \eqref{oper:BR} on $(X,g)$
\begin{eqnarray*}
S_{\lambda}^\delta(\Delta_X)
=\int_0^{\lambda} \Big(1-\frac{\rho^2}{\lambda^2}\Big)^\delta  dE_{\sqrt{\Delta_X}} (\rho).
\end{eqnarray*}
To this end, we define the function
\begin{equation}\label{equ:kerBRR2e}
K^\delta_\lambda(x):=\frac{1}{(2\pi)^2}\int_{\R^2}e^{ix\cdot\xi} \left(1-\frac{|\xi|^2}{\lambda^2}\right)_+^\delta\;d\xi.
\end{equation}
Here  $t_+^\delta=t^\delta$ for $t>0$ and zero otherwise.  From Stein \cite{Stein} and Sogge \cite[Lemma 2.3.3]{Sogge}, we see that the function $K^\delta_\lambda(x)$
satisfies the following property.
\begin{lemma}\cite[Lemma 2.1]{MYZ}\label{lem:BRRnkern}
We can write the function $K^\delta_\lambda(x)$ as
\begin{equation}\label{equ:K-delta}
\begin{split}
&K^\delta_\lambda(x)=K^\delta_\lambda (|x|)\\
&=\frac{\lambda^2a_+(\lambda|x|)e^{i\lambda|x|}}{(1+\lambda|x|)^{\frac{3}2+\delta}}+\frac{\lambda^2a_-(\lambda|x|)e^{-i\lambda |x|}}{(1+\lambda|x|)^{\frac{3}2+\delta}}+O\big(\lambda^2 (1+\lambda|x|)^{-3}\big).
\end{split}
\end{equation}
where $a_\pm$ are bounded from below near infinity and satisfy
\begin{equation}\label{equ:ajbeha}
\left| \frac{d^k}{dr^k}a_\pm(r)\right|\leq C_k r^{-k},\quad\; \forall\;k\in\N.
\end{equation} Here $C_k$ is a constant depending on $k$.
\end{lemma}

For the rest of this subsection, we prove the following proposition, which gives an expression of the kernel of the Bochner-Riesz mean operator that we will study. 

\begin{proposition}\label{prop:BRker}The kernel of the Bochner-Riesz mean operator $S_{\lambda}^\delta(\Delta_X)$
can be written as
\begin{equation}\label{equ:BRkernel}
\begin{split}
&S_{\lambda}^\delta(\Delta_X)\\
=&\sum_{\{j\in\Z: 0\leq|\theta_1-\theta_2+2j\sigma \pi|\leq \pi\}}K^\delta_\lambda(|\mathbf{m}_j|)-\frac1{\pi \sigma}\int_0^\infty K_\lambda^\delta(|\mathbf{n}_s|)A_\sigma(s,\theta_1,\theta_2)\;ds,
\end{split}
\end{equation}
where $K^\delta_\lambda$ is in \eqref{equ:K-delta},  ${\bf m}_j, {\bf n}_s\in\R^2$ are given in \eqref{equ:mjdef} and $A_\sigma(s,\theta_1,\theta_2)$ in \eqref{equ:Asigmadef}.
\end{proposition}

\begin{proof}
By \eqref{oper:BR} and \eqref{equ:spect}, we have
\begin{align*}
S_{\lambda}^\delta(\Delta_X)
=\frac1{4\pi^3i}\int_0^\infty \rho \left(1-\frac{\rho^2}{\lambda^2}\right)_+^\delta
&\Big\{ \sum_{\{j\in\Z: 0\leq|\theta_1-\theta_2+2j\sigma \pi|\leq \pi\}} 
\mathrm{Spect}(\rho, {\bf m}_j)\\
&-\frac1{\pi \sigma}\int_0^\infty \mathrm{Spect}(\rho, {\bf n}_s)
A_\sigma(s,\theta_1,\theta_2)\;ds\Big\}\;d\rho.
\end{align*}
To prove \eqref{equ:BRkernel}, it suffices to show
\begin{equation}\label{id-kernel}
K^\delta_\lambda(|{\bf m}|)= \int_0^\infty \rho \left(1-\frac{\rho^2}{\lambda^2}\right)_+^\delta \mathrm{Spect}(\rho, {\bf m}) \, d\rho,
\end{equation}
where $\mathrm{Spect}(\rho, {\bf m})$ is in \eqref{def:spec}.  In fact, we have
\begin{equation*}
\begin{split}
&\mathrm{Spect}(\rho, {\bf m})=\int_{\R^2} e^{-i\mathbf{m}\cdot\xi}\left(\frac{1}{|\xi|^2-(\rho^2+i0)}-\frac{1}{|\xi|^2-(\rho^2-i0)}\right)\;d\xi\\
&=\lim_{\epsilon\to 0^+}\int_{\R^2} e^{-i{\bf m}\cdot\xi}\Big(\frac{1}{|\xi|^2-(\rho^2+i\epsilon)}-\frac{1}{|\xi|^2-(\rho^2-i\epsilon)}\Big) d\xi\\
&=\lim_{\epsilon\to 0^+} \int_{\R^2} e^{-i{\bf m}\cdot\xi}\Im\Big(\frac{1}{|\xi|^2-(\rho^2+i\epsilon)}\Big)d\xi\\
&=\lim_{\epsilon\to 0^+}\int_{0}^\infty \frac{\epsilon}{(\lambda^2-\rho^2)^2+\epsilon^2} \int_{|\omega|=1} e^{-i\lambda {\bf m}\cdot\omega} d\sigma_\omega  \, \lambda d\lambda\\
&=  \int_{|\omega|=1} e^{-i\rho {\bf m}\cdot\omega} d\sigma_\omega ,
\end{split}
\end{equation*}
where we use the fact that the Poisson kernel is an approximation to the identity which implies that for $m \in L^\infty$ we have
%, for any reasonable function $m(x)$
\begin{equation}
\begin{split}
m(x)&=\lim_{\epsilon\to 0^+}\frac1\pi \int_{\R} \Im\big(\frac{1}{x-(y+i\epsilon)}\big) m(y)dy
\\&=\lim_{\epsilon\to 0^+}\frac1\pi \int_{\R} \frac{\epsilon}{(x-y)^2+\epsilon^2} m(y)dy
\end{split}
\end{equation}
for almost every $x$.

Therefore, the right hand side of \eqref{id-kernel} equals 
\begin{equation*}
\begin{split}
\int_0^\infty \rho \left(1-\frac{\rho^2}{\lambda^2}\right)_+^\delta  \int_{|\omega|=1} e^{-i\rho {\bf m}\cdot\omega} d\sigma_\omega \, d\rho&=\int_{\R^2} e^{-i {\bf m}\cdot\xi} \left(1-\frac{|\xi|^2}{\lambda^2}\right)_+^\delta \, d\xi\\
&=K^\delta_\lambda(|{\bf m}|),
\end{split}
\end{equation*}
which proves \eqref{id-kernel}.

\end{proof}

\subsection{Reduction}

Before proceeding to the estimate and the concrete expressions of kernels involved, we prove a lemma allowing us to reduce the problem to the case with the radius $\sigma > 1$.

To avoid nested sub-indices, we use $K_{\lambda}^{\delta}(r_1,\theta_1,r_2,\theta_2;\sigma)$ to denote the kernel of the Bochner-Riesz kernel $S_{\lambda}^{\delta}(\Delta_X)$ defined in \eqref{oper:BR} with $X = X_{\sigma} =  C(\mathbb{S}^1_{\sigma})$. Then we have the following Lemma relating the Bochner-Riesz kernel of $X_{\sigma}$ and that of $X_{\frac{\sigma}{2}}$.
%$S_{\lambda}^{\delta}(\Delta_{ C(S_{\frac{\sigma}{2}}^1) }$.

\begin{lemma} \label{lemma:sigma-reduction}
With notations as above, we have\footnote{It's not guaranteed, in fact not true, that those kernels can be evaluated pointwise, hence the identity is understood as viewing the second term on the right hand side as the pull-back of that distribution under that rotation map.  In addition, the right hand side originally is only a distribution on $\mathbb{S}_{\sigma}^1 = \mathbb{R}/\sigma \mathbb{Z}$, but it descends to a distribution on $\mathbb{S}_{\sigma/2}^1 = \mathbb{R}/\frac{\sigma}{2} \mathbb{Z}$ due to its invariance under $\theta_i \to \theta_i + \frac{\sigma}{2}$ and the equality makes sense. The same interpretation for other kernels is applied below.}
\begin{equation} \label{eq:sigma-reduction-BR-kernel}
\begin{split}
&K_{\lambda}^{\delta} (r_1,\theta_1,r_2,\theta_2;\frac{\sigma}{2}) \\
&= \frac{1}{2}
\Big( K_{\lambda}^{\delta} (r_1,\theta_1,r_2,\theta_2;\sigma)  + K_{\lambda}^{\delta} (r_1,\theta_1+\frac{\sigma}{2}, r_2,\theta_2+\frac{\sigma}{2};\sigma)  \Big)
\end{split}
\end{equation}
\end{lemma}
\begin{remark}
One can see from the proof below, the factor $\frac{1}{2}$ here is not important and can be replaced by any $\frac{1}{k}$ where $k\in\N$. And the proof is a standard trick that when we have a group $G$ acting nicely on a space $X$, then objects on $X$ descends to the same type object on $X/G$ after averaging over orbits of the $G$-action.
\end{remark}

\begin{proof}
From the definition in \eqref{oper:BR}, the conclusion will follow if we can prove corresponding identity for the spectral measure:
\begin{equation} \label{eq:sigma-reduction-spectral-measure}
\begin{split}
&dE_{\sqrt{ \Delta_{X_{\sigma/2} } } } (r_1,\theta_1,r_2,\theta_2) \\
&= \frac{1}{2}
\Big( dE_{\sqrt{ \Delta_{X_{\sigma} } } } (r_1,\theta_1,r_2,\theta_2)  + dE_{\sqrt{ \Delta_{X_{\sigma} } } } (r_1,\theta_1+\frac{\sigma}{2}, r_2,\theta_2+\frac{\sigma}{2};\sigma)  \Big).
\end{split}
\end{equation}
This will follow from the same type identity for $(\Delta_{X}-(\rho^2 \pm i0))^{-1}$ since the spectral measure is just the difference between the outgoing and incoming resolvents.

Those identities for resolvents in turn follows from the identity for the Schr\"odinger propagator
\begin{equation} \label{eq:sigma-reduction-Schrodinger-propgator}
\begin{split}
&e^{it\Delta_{X_{\sigma/2}}}(r_1,\theta_1,r_2,\theta_2) \\
&= \frac{1}{2}
\Big( e^{it\Delta_{X_{\sigma}}} (r_1,\theta_1,r_2,\theta_2)  + e^{it\Delta_{X_{\sigma}}} (r_1,\theta_1+\frac{\sigma}{2}, r_2,\theta_2+\frac{\sigma}{2})  \Big),
\end{split}
\end{equation}
in combination with \eqref{res+} and its analogue for $(\Delta_{X}-(\rho^2-i0))^{-1}$.

On the other hand, \eqref{eq:sigma-reduction-Schrodinger-propgator} holds by considering the corresponding solutions to PDEs.
We let the right hand side of \eqref{eq:sigma-reduction-Schrodinger-propgator} act on a function $u_0 \in \mathcal{S}(X_{\sigma/2})$ and denote the function by $u(t,r_1,\theta_1)$. Then it solves the equation
\begin{align}
\begin{split}
& (i\partial_t + \Delta_{X_{\sigma/2}}) u =0,
\\ & u|_{t=0} = u_0.
\end{split}
\end{align}
The first line follows from the definition of the Schr\"odinger propagators on the right hand side and the fact that the local expression of $\Delta_{X_\sigma}$ and $\Delta_{X_{\sigma/2}}$ are the same.
More precisely, we view $u$ (and consequently $\Delta_{X_\sigma}u$ as well)  viewed as a function (or distribution) on $X_{\sigma}$ invariant under $\theta \to \theta + \frac{\sigma}{2}$ and this coincide with $\Delta_{X_{\sigma/2}}u$ when it is viewed as a function (or distribution) on $X_{\sigma}$ in the same way. Then the function that they descends to on $X_{\sigma/2}$ will be the same as well.

The second line follows from the fact that both terms on the right hand side are the $\delta$ function on the diagonal when $t=0$, noticing that $\delta(\theta_1+\frac{\sigma}{2}- (\theta_2 + \frac{\sigma}{2}) ) = \delta(\theta_1-\theta_2)$, hence the output restricted to $t=0$ is still $u_0$. By the uniqueness of solutions to the Schr\"odinger equation, we know \eqref{eq:sigma-reduction-Schrodinger-propgator} holds.

\end{proof}

In the rest of this paper, without loss of generality, we always assume $\sigma>1$, which simplifies expressions involving summation over $j$ such that $|\theta_1-\theta_2+2j\sigma \pi| \leq \pi$ since it will contain at most three terms then.
But one should only view this as a formal simplification as the summation is now packaged into Lemma~\ref{lemma:sigma-reduction}.
%which makes the situation become simple.

%%%%%%%%%%%%%%%%%%%%%%%%%%%%%%%%%%%%%%%%%%%%%%%%%%%%%%%%%%%%%%%%%%%%%%%%%%%%%%%%%%%%%%%%%%%%%%%%%%%%%%%%%%%%%%%%%%%%%%%%%%%%%%%%%%%%%%%%%%%%%%

%%%%%%%%%%%%%%%%%%%%%%%%%%%%%%%%%%%                                %%%%%%%%%%%%%%%%%%%%%%%%%%%%%%%%%%%%%%%%%%%%%%%%%%%%%%%%%%%%%%%%%%%%%%

%%%%%%%%%%%%%%%%%%%%%%%%%%%%%%%%%%%%%%%%%%%%%%%%%%%%%%%%%%%%%%%%%%%%%%%%%%%%%%%%%%%%%%%%%%%%%%%%%%%%%%%%%%%%%%%%%%%%%%%%%%%%%%%%%%%%%%%%%%%%%%

\section{Proof of Theorem \ref{thm:main}}\label{Sec:mainthm}

In this section, we prove Theorem \ref{thm:main}. 
The idea is to apply the oscillatory integral theory of Stein \cite{Stein} and H\"ormander \cite{Hor} to study the Bochner-Riesz means kernel \eqref{equ:BRkernel}, but we need new estimates for diffraction terms caused by the vertex of the cone.
%idea to overcome the singularity and diffraction effects. 

To prove Theorem \ref{thm:main}, due to \eqref{equ:spect-scaling}, which reflects the scaling invariance of $X$, it suffices to prove \eqref{est:BRcone} with $\lambda=1$:
\begin{equation}\label{est:BRcone'}
\big\|S_1^\delta(\Delta_X)f\big\|_{L^p(X)}\leq C\|f\|_{L^p(X)}, \quad \delta>\delta_c(p, 2).
\end{equation}
To this end, we first use Proposition \ref{prop:BRker} and \eqref{equ:K-delta} to divide the kernel into several pieces: for $x=(r_1,\theta_1),\;y=(r_2,\theta_2)$ in $X=C(\mathbb{S}_\sigma^1)$,
\begin{align}\label{equ:SRdelkern}
\nonumber &S_1^\delta(\Delta_X)(x,y)\\ 
=&\sum_{\pm}\big(G_{m}^\pm+D_{m}^\pm\big)(\delta;r_1,\theta_1;r_2,\theta_2)
+\big(G_e+D_e\big)(\delta;r_1,\theta_1;r_2,\theta_2),
\end{align}
where
\begin{align}\label{equ:Gjktermdef}
G_{m}^\pm(\delta;r_1,\theta_1;r_2,\theta_2)=&\sum_{\{j\in\Z: 0\leq|\theta_1-\theta_2+2j\sigma \pi|\leq \pi\}}\frac{a_\pm(d_j)e^{\pm id_j}}{(1+d_j)^{\frac{3}2+\delta}},
\\\label{equ:Gj3termdef}
G_{e}(\delta;r_1,\theta_1;r_2,\theta_2)=&\sum_{\{j\in\Z: 0\leq|\theta_1-\theta_2+2j\sigma \pi|\leq \pi\}}b(d_j),
\end{align}
with 
\begin{align} \label{eq:def-dj}
d_j=d_j(r_1,r_2,\theta_1,\theta_2)=|{\bf m}_j|=&\sqrt{r_1^2+r_2^2-2\cos(\theta_1-\theta_2+2j\sigma\pi)r_1r_2},
\end{align}
\begin{align}\label{eq:b,a-symbol-bounds}
|b(r)|\leq C(1+r)^{-3},\quad \Big|\frac{d^k}{dr^k}a_\pm(r)\Big|\leq C_k(1+r)^{-k},\quad \forall\;k\in\N,
\end{align}
and
\begin{align}\label{equ:Dktermdef}
D_{m}^\pm(\delta;r_1,\theta_1;r_2,\theta_2)=&-\frac1{\pi \sigma}\int_0^\infty \frac{a_\pm(d_s)e^{\pm id_s}}{(1+d_s)^{\frac{3}2+\delta}}A_\sigma(s,\theta_1,\theta_2)\;ds,\\\label{equ:D3termdef}
D_{e}(\delta;r_1,\theta_1;r_2,\theta_2)=&-\frac1{\pi \sigma}\int_0^\infty b(d_s)A_\sigma(s,\theta_1,\theta_2)\;ds,\\
\label{equ:dsdefr12}
d_s=d_s(r_1,r_2,\theta_1,\theta_2)&=|{\bf n}_s|=\sqrt{r_1^2+r_2^2+2r_1r_2\cosh(s)}. 
\end{align}

Therefore, to prove \eqref{est:BRcone'}, it suffices to show 
\begin{equation}\label{est:BRTK}
\big\|T_Kf\big\|_{L^p(X)}\leq C\|f\|_{L^p(X)},\quad K\in\big\{G_m^\pm, D_m^\pm, G_e, D_e\big\}.
\end{equation}
where
$$T_Kf(x)=\int_0^\infty\int_{\Sa} K(\delta;r_1,\theta_1;r_2,\theta_2)f(r_2,\theta_2)\;d\theta_2\;r_2dr_2.$$
So in the rest of this section, we are going to prove  \eqref{est:BRTK}. 
We remark here that, for given $\sigma>0$,  the summation of $j$ in the terms $G_m^\pm$ and $G_e$ is finite and bounded by $O\big(1+\tfrac1\sigma\big).$ Under the assumption that $\sigma>1$, we only need to consider $j=0$ or $\pm 1$.

\subsection{The contribution of $T_{G_e}$ and $T_{D_e}$} 
The contribution from these two terms are relatively easy to treat since they have sufficient decay for \eqref{est:BRTK}.
%The two terms are easy to be proved since the decay is enough for \eqref{est:BRTK}.  

 Recall
$$T_{G_e}f(x)=\int_0^\infty\int_{\Sa}\sum_{\{j\in\Z: 0\leq|\theta_1-\theta_2+2j\sigma \pi|\leq \pi\}}b(d_j)  f(r_2,\theta_2)\;d\theta_2\;r_2dr_2.$$
Noting that for given $\sigma>1$, the sum actually only have at most one term (it could be empty) with $j \in \{-1,0,1\}$ depending on $\theta_1-\theta_2$. 
%the number of terms in the summation over $j$ above is finite and bounded by $O\big(1+\tfrac1\sigma\big)$

Using the fact that $|b(d_j)|\leq C(1+d_j)^{-3}$ in \eqref{eq:b,a-symbol-bounds}, we obtain 
\begin{align*}
&\big\|T_{G_e}f\big\|_{L^p(X)}\\
\leq C &\sup_{\{j\in\Z: 0\leq|\theta_1-\theta_2+2j\sigma \pi|\leq \pi\}} \left\|\int_0^\infty \int_0^{2\sigma\pi}\frac{|f(r_2,\theta_2)|}{(1+d_j(r_1,\theta_1; r_2,\theta_2))^3}d\theta_2\;r_2dr_2\right\|_{L^p_{\theta_1}([0,2\sigma\pi],L^p(\R^+,r_1dr_1)}.
\end{align*}
To bound it, we need a lemma called general Young's inequality whose proof will be given in Appendix~\ref{Sec:appendix}.
\begin{lemma}[General Young's inequality]\label{lem:GYoung}
 Let $\sigma>0$, $\theta_0\in\R$  and $1\leq p\leq+\infty$. Then, for any $\alpha>2$, there exists a constant $C$, depending on $\sigma$ but not $\theta_0$, such that 
\begin{align}\nonumber
&\left\|\int_0^\infty \int_{0}^{2\sigma\pi}\frac{|f(r_2,\theta_2)|}{\big(1+\sqrt{r_1^2+r_2^2-2\cos(\theta_1-\theta_2-\theta_0)r_1r_2}\big)^{\alpha}}\;d\theta_2r_2dr_2\right\|_{L^p_{\theta_1}([0,2\sigma\pi], L^p(\R^+,r_1dr_1))}\\\label{equ:GYoung}
\leq& C\|f\|_{L^p(X)}.
\end{align}
\end{lemma}
Thus we can apply Lemma \ref{lem:GYoung} with $\theta_0=2j\sigma \pi$ to obtain 
$$\big\|T_{G_e}f\big\|_{L^p(X)}\leq C_\sigma \|f\|_{L^p(X)}.$$
Next we consider 
$$T_{D_e}f(x)=-\frac1{\pi \sigma}\int_0^\infty\int_{\Sa}\int_0^\infty b(d_s)A_\sigma(s,\theta_1,\theta_2)\;ds  f(r_2,\theta_2)\;d\theta_2\;r_2dr_2.$$
Recalling \eqref{equ:Asigmadef}, we have  
\begin{equation}\label{equ:Asigmaintbd}
\int_0^\infty \big|A_\sigma(s,\theta_1,\theta_2)\big|\;ds\leq C,
\end{equation} 
 if we could prove
\begin{equation}\label{0-infty}
\begin{split}
 \int_0^\infty  \Big|\Im \Big(\frac{e^{i\theta}}{e^{s}-e^{i\theta}}\Big)\Big| ds\leq C
\end{split}
\end{equation}
where
$$\theta=\frac1\sigma(\pi-(\theta_1-\theta_2))\quad \text{or}\quad \frac1\sigma(-\pi-(\theta_1-\theta_2)),$$ and $C$ is a constant independent of $\theta$. To see this, we first have
\begin{equation*}
\begin{split}
 \int_1^\infty \Big|\Im \Big(\frac{e^{i\theta}}{e^{s}-e^{i\theta}}\Big)\Big| ds\leq  \int_1^\infty e^{-s/2} ds\leq C.
\end{split}
\end{equation*}
On the other hand, we have 
\begin{equation}
\begin{split}
\Im \Big(\frac{e^{i\theta}}{e^{s}-e^{i\theta}}\Big)&=\Im \Big(\frac{(\cos\theta+i\sin\theta)(e^s-\cos\theta+i\sin\theta)}{(e^{s}-\cos \theta)^2+\sin^2\theta}\Big)\\
&=
\frac{e^s\, \sin\theta }{(e^{s}-\cos \theta)^2+\sin^2\theta}.
\end{split}
\end{equation}
Therefore we obtain 
\begin{equation*}
\begin{split}
 \int_0^1 \Big|\Im \Big(\frac{e^{i\theta}}{e^{s}-e^{i\theta}}\Big)\Big| ds&\leq  \int_0^1 \frac{e^s\, |\sin\theta| }{(e^{s}-\cos \theta)^2+\sin^2\theta} ds\\
 &\leq \int_0^3 \frac{ |\sin\theta| }{s^2+\sin^2\theta} ds\leq  \int_0^\infty \frac{1}{s^2+1} ds\leq C.
\end{split}
\end{equation*}
By using \eqref{equ:Asigmaintbd}, we see
\begin{align*}
\big\|T_{D_e}f\big\|_{L^p(X)}\leq&C\left\|\int_0^\infty\int_{\Sa}\sup_{s\geq0} |b(d_s)|\cdot  \big|f(r_2,\theta_2)\big|\;d\theta_2\;r_2dr_2\right\|_{L^p_x(X)}.
\end{align*}
Noting that $d_s\geq (r_1+r_2)$, and so
\begin{equation}\label{eq:bdsest}
 |b(d_s)|\leq C(1+d_s)^{-3}\leq C(1+r_1+r_2)^{-3},
\end{equation}
hence 
\begin{align*}\nonumber
\big\|T_{D_2}f\big\|_{L^p(X)}\leq&C\left\|\int_0^\infty\int_{\Sa} \frac{\big|f(r_2,\theta_2)\big|}{(1+r_1+r_2)^3}\;d\theta_2\;r_2dr_2\right\|_{L^p_x(X)},
\end{align*}
we further use Lemma \ref{lem:logYoung}, Minkowski's inequality, and H\"older's inequality to obtain 
\begin{align}\nonumber
\big\|T_{D_2}f\big\|_{L^p(X)}\leq&C\left\|\int_0^\infty\int_{\Sa} \frac{\big|f(r_2,\theta_2)\big|}{(1+r_1+r_2)^3}\;d\theta_2\;r_2dr_2\right\|_{L^p_x(X)}\\\nonumber
\leq&C\int_{\Sa}\left\| \int_0^\infty \frac{\big|f(r_2,\theta_2)\big|}{(1+r_1+r_2)^3}\;r_2dr_2\right\|_{L^p(\R^+,r_1dr_1)}\;d\theta_2\\\label{equ:TD2termest}
\leq&C\|f\|_{L^p(X)}.
\end{align}
In sum, we have proved  \eqref{est:BRTK} for $K=G_e$ and $D_e$.

\subsection{The contribution of $T_{G_m^\pm}$} In this subsection, we aim to show
\begin{equation}\label{equ:TG1pmterm}
\|T_{G_m^\pm}f\|_{L^p(X)}\leq C\|f\|_{L^p(X)},
\end{equation}
where
 $$T_{G_m^\pm}f(x)=\int_0^\infty\int_{\Sa}\sum_{\{j\in\Z: 0\leq|\theta_1-\theta_2+2j\sigma \pi|\leq \pi\}} \frac{a_\pm(d_j)e^{\pm id_j}}{(1+d_j)^{\frac{3}2+\delta}}  f(r_2,\theta_2)\;d\theta_2\;r_2dr_2.$$
 Notice that
 $$\theta_1-\theta_2\in (-2\pi\sigma,2\pi\sigma), $$
 therefore we only need to consider $j=0$ and $j=\pm1$ due to $\sigma>1$. For $j=0, \pm1$, define the sets
 \begin{equation}\label{def:Aj}
 A_j:=\{(\theta_1,\theta_2)\in[0,2\sigma\pi)^2: 0\leq|\theta_1-\theta_2+2j\sigma \pi|\leq \pi\},
 \end{equation}
 and we define the operator $ T_{G_{mj}^\pm}$ as follows
 \begin{equation}\label{equ:TG1jdef}
 T_{G_{mj}^\pm}f(x)=\int_0^\infty\int_{\Sa} G_{mj}^\pm(\delta;r_1,\theta_1;r_2,\theta_2)f(r_2,\theta_2)\;d\theta_2\;r_2dr_2.
 \end{equation}
 where
  \begin{equation}\label{equ:G1jdef}
G_{mj}^\pm(\delta;r_1,\theta_1;r_2,\theta_2)=\frac{a_\pm(d_j)e^{\pm id_j}}{(1+d_j)^{\frac{3}2+\delta}} \mathbbm{1}_{A_j}(\theta_1, \theta_2).
 \end{equation}
 To show \eqref{equ:TG1pmterm},  it is enough to prove
\begin{equation}\label{equ:TG!jest}
\|T_{G_{mj}^\pm}f\|_{L^p(X)}\leq C_j\|f\|_{L^p(X)},\quad j=0, \pm 1.
\end{equation}
By partition of unity
\begin{equation}\label{equ:paruni}
\psi_0(r)=1-\sum_{k\geq1}\psi(2^{-k}r),\quad \psi\in C_c^\infty\big(\big[\tfrac34,\tfrac83\big]\big),
\end{equation}
we decompose 
\begin{equation}\label{equ:G1pmkern}
G_{mj}^\pm(\delta;r_1,\theta_1;r_2,\theta_2)=\sum_{k\geq0}G_{mj,k}^\pm(\delta;r_1,\theta_1;r_2,\theta_2),
\end{equation}
with
\begin{equation}\label{equ:G1kpern}
G_{mj,k}^\pm(\delta;r_1,\theta_1;r_2,\theta_2)=G_{mj}^\pm(\delta;r_1,\theta_1;r_2,\theta_2)\times\begin{cases}
\psi(2^{-k}d_j),\quad\text{if}\;k\geq1,\\
\psi_0(d_j),\quad\text{if}\;k=0.
\end{cases}
\end{equation}
Thus we prove \eqref{equ:TG!jest} if we could prove 
\begin{proposition}\label{prop:Tg1kest}
Let $T_{G_{mj,k}^\pm}$ be the operator associated with the kernel $G_{mj,k}^\pm(\delta;r_1,\theta_1;r_2,\theta_2)$ where $j=0, \pm 1$ and $k\geq0$.
Then, there holds 
\begin{equation}\label{equ:TG1kest}
\|T_{G_{mj,k}^\pm}\|_{L^p(X)\to L^p(X)} \leq C_j\times \begin{cases}
2^{k(\delta_c(p,2)-\delta)},\quad\text{if}\;p>4,\\
2^{-k\delta},\quad\text{if}\;p=2,
\end{cases}
\end{equation}
for $\delta>\delta_c(p,2)$ is as in \eqref{equ:BRconj}.

\end{proposition}

\begin{proof}
For simplicity, we only consider the $+$ case since the $-$ case can be treated similarly. 
We first consider the case $k=0$. Due to $d_j\lesssim 1$, we use Lemma \ref{lem:GYoung} again to obtain 
\begin{align*}
&\|T_{G_{mj,0}^+}f(r_1,\theta_1)\|_{L^p_{\theta_1}([0,2\sigma\pi],L^p(r_1dr_1))}\\
\lesssim&\left\|\int_0^\infty\int_0^{2\sigma\pi}\psi_0(d_j)\frac{|f(r_2,\theta_2)|}{(1+d_j)^{\frac{3}2+\delta}}\frac1{(1+d_j)^{\frac{3}2}}\;d\theta_2\;r_2dr_2  \right\|_{L^p_{\theta_1}([0,2\sigma\pi],L^p(r_1dr_1))}\\
\lesssim&\|f\|_{L^p(X)}.
\end{align*}
Next we consider the case $k\geq1$. We only provide the detail proof for the case $j=0$, since the same argument also works for $j=\pm1$.
When $j=0$, then
 \begin{equation}\label{def:A0}
 A_0=:\{(\theta_1,\theta_2)\in[0,2\sigma\pi)^2: 0\leq|\theta_1-\theta_2|\leq \pi\}.
 \end{equation}
 From now on, we drop the index $j$ in the notation $G_{mj,k}^\pm$ for simplicity. 
 Notice that the kernel $G_{m,k}^\pm(\delta;r_1,\theta_1;r_2,\theta_2)$ is a function of $\theta_1-\theta_2$, to prove \eqref{equ:TG1kest}, we may assume that
\begin{equation}\label{equ:assumpfcond}
\text{supp}\;f(r_2,\theta_2)\subset\big\{(r_2,\theta_2):\;r_2\in\R^+,\;\theta_2\in[0,\epsilon]\big\},\quad 0<\epsilon\ll1,
\end{equation} 
 due translation invariance in integral in $\theta$ and we can write $f$ as a finite sum of functions supported $[0,\epsilon]$ and its translated copies.
 %to compactness in $\theta_2$
 %the finite covering lemma and 
 
 Observe that if $\theta_2\in[0,\epsilon]$, then the characteristic function 
\begin{equation}\label{eq:chara}
\chi_{A_0}(\theta_1, \theta_2)=\mathbbm{1}_{[0,\pi]}(|\theta_1-\theta_2|)= \begin{cases}
1,\quad\text{when}\quad \theta_1\in[0,\pi],\\
0,\quad\text{when}\quad \theta_1\in[\pi+\epsilon,2\sigma\pi].
\end{cases}
\end{equation}
So we are reduced to consider the following three cases: 
\begin{equation}\mathrm{(i)} \, \theta_1\in[\pi+\epsilon,2\sigma\pi]; \quad \mathrm{(ii)} \, \theta_1\in[0,\pi]; \quad \mathrm{(iii)} \, \theta_1\in[\pi,\pi+\epsilon].
\end{equation}
The first case is trivial since the kernel vanishes. In the second case that $\theta_1\in[0,\pi]$, we can replace  the characteristic function $\chi_{A_0}(\theta_1, \theta_2)$ by $\mathbbm{1}_{[0,\pi]}(\theta_1)$ due to \eqref{eq:chara}, 
where we use the fact that $d_0=\sqrt{r_1^2+r_2^2-2r_1r_2\cos(\theta_1-\theta_2)}=|x-y|$ to see 
the kernel
$$\mathbbm{1}_{[0,\pi]}(\theta_1)\frac{a_+(d_0)e^{id_0}}{(1+d_0)^{\frac{3}2+\delta}}\psi(2^{-k}d_0)=\mathbbm{1}_{[0,\pi]}(\theta_1)\frac{a_+(|x-y|)e^{i|x-y|}}{(1+|x-y|)^{\frac{3}2+\delta}}\psi(2^{-k}|x-y|).$$
%where $d_0=\sqrt{r_1^2+r_2^2-2r_1r_2\cos(\theta_1-\theta_2)}=|x-y|$. 
Therefore, due to \eqref{equ:assumpfcond}, we obtain 
\begin{equation*}
\begin{split}
&\left\|\mathbbm{1}_{[0,\pi]}(\theta_1)T_{G_{m,k}^+}(r_1,\theta_1)f\right\|_{L^p_{\theta_1}([0,2\sigma\pi],L^p(r_1dr_1))}\\
&\leq \Big\|\int_{\R^2}\psi(2^{-k}|x-y|)e^{i|x-y|} (1+|x-y|)^{-3/2-\delta} a_+(|x-y|)f(y) dy \Big\|_{L^p(\R^2)},
\end{split}
\end{equation*}
and the right hand side can be estimated by standard facts for oscillation integrals (e.g. \cite[Lemma 2.3.4]{Sogge}).
To do this, we define 
$$K_{G^+_{m,k}}(x)=\psi(2^{-k}|x|)e^{i|x|} (1+|x|)^{-3/2-\delta} a_+(|x|).$$
We observe that the $L^p-L^p$ operator norm of convolution with $K_{G^+_{m,k}}$ equals to the norm of convolution with the dilated kernel
\begin{equation*}
\begin{split}
\tilde{K}_{G^+_{m,k}}(x)=2^{2k}K_{G^+_{m,k}}(2^k x)&=2^{k(\frac12-\delta)} \psi(|x|)e^{i2^k |x|} (2^{-k}+|x|)^{-3/2-\delta} a_+(2^k |x|)\\
&=:2^{k(\frac12-\delta)}b(2^k, x)e^{i2^k |x|}.
\end{split}
\end{equation*}
Recalling that $a_+$ satisfies \eqref{eq:b,a-symbol-bounds}, we see that $\big|(\frac{\partial}{\partial x})^\alpha b(2^k, x)\big|\leq C_\alpha$. Then 
\cite[Lemma 2.3.4 or $(2.3.9')$]{Sogge} implies
$$\|\tilde{K}_{G^+_{m,k}}(x)\ast f\|_{L^p(\R^2)}\leq C 2^{k(\frac12-\delta-\frac2p)}\|f\|_{L^p(\R^2)}=C2^{k[\delta_c(p,2)-\delta]}\|f\|_{L^p}, \quad p>4.$$
For $p=2$, we use \cite[Theorem 2.1.1]{Sogge} to obtain 
$$\|\tilde{K}_{G^+_{m,k}}(x)\ast f\|_{L^2(\R^2)}\leq C 2^{k(\frac12-\delta)} 2^{-\frac k2} \|f\|_{L^2(\R^2)}=C 2^{-k\delta} \|f\|_{L^2(\R^2)}.$$
Hence, we have proved 
\begin{equation*}
\begin{split}
&\left\|\mathbbm{1}_{[0,\pi]}(\theta_1)T_{G_{m,k}^+}(r_1,\theta_1)f\right\|_{L^p_{\theta_1}([0,2\sigma\pi],L^p(r_1dr_1))}\\
&\qquad \leq C\|f\|_{L^p(X)}\times \begin{cases}
2^{k(\delta_c(p,2)-\delta)},\quad\text{if}\;p>4,\\
2^{-k\delta},\quad\text{if}\;p=2,
\end{cases}
\end{split}
\end{equation*}
which gives 
\eqref{equ:TG1kest} in this case.\vspace{0.2cm}

We finally consider the third case $\theta_1\in[\pi,\pi+\epsilon].$
In this case, due to that $\sigma>1$ and the presence of the factor $\mathbbm{1}_{[\pi,\pi+\epsilon]}(\theta_1)$ the integral over $\theta_1$ only happens over $[0,2\pi]$ and the estimate can be reduced to the $\R^2$-case:
\begin{equation*}
\begin{split}
&\left\|\mathbbm{1}_{[\pi,\pi+\epsilon]}(\theta_1)[T_{G_{m,k}^+}f] (r_1,\theta_1)\right\|_{L^p_{\theta_1}([0,2\sigma\pi],L^p(r_1dr_1))}\\
&\lesssim \left\|\mathbbm{1}_{[\pi,\pi+\epsilon]}(\theta_1)[T_{G_{m,k}^+}f] (r_1,\theta_1)\right\|_{L^p_{\theta_1}([0,2\pi],L^p(r_1dr_1))}\\
&=\left\|\mathbbm{1}_{[\pi,\pi+\epsilon]}(\theta_1)[T_{G_{m,k}^+}f] (r_1,\theta_1)\right\|_{L^p(\R^2)},
\end{split}
\end{equation*}
which agrees with the operator on $\R^2$ studied in \cite[Case 3 of Section 4]{MYZ}. Thus we can modify the argument to estimate it. 
Now we sketch the main steps for reader's convenience. 

Similarly as above, we define
$$b(2^k, |x-y|)=\psi(|x-y|)(2^{-k}+ |x-y|)^{-3/2-\delta} a(2^k|x-y|),$$
where $a$ also satisfies \eqref{eq:b,a-symbol-bounds}.
Thus for every $N\geq 0$, we see that 
\begin{equation}\label{est:bk}
\big|\partial_r^N b(2^k, r)\big|\leq C_N
\end{equation} due to  the compact support of $\psi$.
To prove \eqref{equ:TG1kest} in this case, from \eqref{equ:G1jdef}, \eqref{equ:G1kpern} and $d_0=|x-y|$, we first
recall the kernel of $T_{G_{m,k}}$ (we drop off the $+$ now)
$$K_{G_{m,k}}(r_1,r_2;\theta_1-\theta_2)=\psi(2^{-k}|x-y|)e^{i|x-y|} (1+ |x-y|)^{-3/2-\delta} a(|x-y|)
\mathbbm{1}_{[0,\pi]}(|\theta_1-\theta_2|),$$
and we rewrite
\begin{equation}\label{tildeK}
\begin{split}
&\tilde{K}_{G_{m,k}}(r_1,r_2;\theta_1-\theta_2)=K_{G_{m,k}}(2^kr_1,2^kr_2;\theta_1-\theta_2)
\\&=\psi(|x-y|)e^{i2^k |x-y|} (1+2^k |x-y|)^{-3/2-\delta} a(2^k|x-y|)\mathbbm{1}_{[0,\pi]}(|\theta_1-\theta_2|)\\
&=2^{-k(\frac32+\delta)} b(2^k, |x-y|)e^{i2^k |x-y|}\mathbbm{1}_{[0,\pi]}(|\theta_1-\theta_2|).
\end{split}
\end{equation}
We use the dilated kernel $\tilde{K}_{G_{m,k}}$ to define the operator
\begin{equation}\label{tildeT}
\begin{split}
\tilde{T}_{G_{m,k}} f &=\int_0^\infty\int_0^{2\pi} \tilde{K}_{G_{m,k}}(r_1,r_2;\theta_1-\theta_2)\eta(\theta_1-\theta_2) f(r_2,\theta_2) \,d\theta_2 \,r_2 dr_2,
\end{split}
\end{equation}
where $ \eta(s) \in C_c^\infty ([\pi-2\epsilon,\pi+2\epsilon])$ such that $\eta(s)=1$ when $s\in[\pi-\epsilon,\pi+\epsilon]$.
We will see that
\begin{equation*}
\begin{split}
\mathbbm{1}_{[\pi,\pi+\epsilon]}(\theta_1)T_{G_{m,k}} &\big(\mathbbm{1}_{[0,\epsilon]}(\theta_2) f(r_2,\theta_2)\big)(r_1,\theta_1)\\
&=2^{2k}\mathbbm{1}_{[\pi,\pi+\epsilon]}(\theta_1)\tilde{T}_{G_{m,k}}\big(\mathbbm{1}_{[0,\epsilon]}(\theta_2) f(2^kr_2,\theta_2)\big)(2^{-k}r_1,\theta_1).
\end{split}
\end{equation*}
To prove \eqref{equ:TG1kest}  in this case, recalling the index $\delta_c(2,p)$ in \eqref{equ:BRconj},
 it suffices to prove the following proposition.
 \begin{proposition}\label{prop:key} Let $\tilde{T}_{G_{m,k}} $ be the operator defined in \eqref{tildeT} with kernel \eqref{tildeK}. Then
 \begin{equation}\label{est:TGj'2}
\big\|\tilde{T}_{G_{m,k}} f\big\|_{L^2(\R^2)} \leq C2^{-k(\frac32+\delta)} 2^{-\frac{k}{2}}\|f\|_{L^{2}(\R^2)},
\end{equation}
and
\begin{equation}\label{est:TGj'4}
\big\|\tilde{T}_{G_{m,k}} f\big\|_{L^p(\R^2)} \leq C2^{-k(\frac32+\delta)} 2^{-\frac{2k}{p}}\|f\|_{L^{p}(\R^2)},\quad \text{for}\quad p>4,
\end{equation}
where $f$ satisfies \eqref{equ:assumpfcond}, i.e. $f=\mathbbm{1}_{[0,\epsilon]}(\theta) f$.
\end{proposition}

\begin{remark}
We replace $L^p(X)$ by the space $L^p(\R^2)$ due to the characteristic functions $\mathbbm{1}_{[\pi,\pi+\epsilon]}(\theta_1)$, $\mathbbm{1}_{[0,\epsilon]}(\theta_2)$ and $\sigma>1$.
\end{remark}

\begin{proof} We apply the classical harmonic analysis techniques (e.g. bilinear estimate and Fourier multiplier) to prove this Proposition. 
We first write the operator
\begin{equation*}
\begin{split}
\tilde{T}_{G_{m,k}} f(r_1,\theta_1) &=\int_0^\infty\int_0^{2\pi} \tilde{K}_{G_{m,k}}(r_1,r_2;\theta_1-\theta_2)\eta(\theta_1-\theta_2) f(r_2,\theta_2) \,d\theta_2 \,r_2 dr_2\\
&\triangleq\int_0^\infty[\tilde{T}_{G_{m,k},r_1,r_2} f(r_2, \cdot)]( \theta_1) \,r_2 dr_2
\end{split}
\end{equation*}
where the operator $\tilde{T}_{G_{m,k},r_1,r_2}$ is given by
\begin{equation}\label{Tr1r2}
[\tilde{T}_{G_{m,k},r_1,r_2} g(\cdot)]( \theta_1)=\int_0^{2\pi} \tilde{K}_{G_{m,k}}(r_1,r_2;\theta_1-\theta_2) \eta(\theta_1-\theta_2) g(\theta_2) d\theta_2.
\end{equation}
We aim to estimate
\begin{equation}\label{q-q0}
\begin{split}
&\big\|\tilde{T}_{G_{m,k}} f\big\|_{L^p(\R^2)} \leq  \int_0^\infty \Big\| [\tilde{T}_{G_{m,k},r_1,r_2} f(r_2, \cdot)]( \theta_1) \Big\|_{L^p(\R^2)} \,r_2 dr_2 .
\end{split}
\end{equation}

We first prove \eqref{est:TGj'2} by considering $p=2$. On the support of $\tilde{K}_{G_{m,k}}$, we note that $\theta_1\in[\pi-2\epsilon, \pi+2\epsilon]$ and $\theta_2\in[0, \epsilon)$ with $0<\epsilon\ll1$,
hence we can extend the function so that $\theta_j, (j=1,2)$ as variables in $\R$. So we can regard $\tilde{T}_{G_{m,k},r_1,r_2} $ as a convolution kernel in $\theta$ and use Plancherel's theorem to obtain 
\begin{equation}\label{2-2}
\begin{split}
& \Big\| [\tilde{T}_{G_{m,k},r_1,r_2} f(r_2, \cdot)]( \theta_1) \Big\|_{L^2_{\theta_1}([0,2\pi))}= \Big\| [\tilde{T}_{G_{m,k},r_1,r_2} f(r_2, \cdot)]( \theta_1) \Big\|_{L^2_{\theta_1}(\R)}\\
&= \Big\| \widehat{\tilde{K}_{G_{m,k}} \eta}(\zeta) \hat{f}(\zeta) \Big\|_{L^2 (d\zeta)}\leq C 2^{-k(\frac32+\delta)} (2^kr_1r_2)^{-1/2} \|f(\theta_2)\|_{L^2_{\theta_2}(\R)},
\end{split}
\end{equation}
if we could prove that the corresponding Fourier multiplier satisfies
\begin{align}\nonumber
  &\Big|\int_{\pi-2\epsilon}^{\pi} e^{i\zeta\theta+i2^kd(r_1,r_2; \theta)}
 b(2^k, d(r_1,r_2; \theta))\eta(\theta)
\;d\theta\Big|\\\label{equ:kernconfrou}
\lesssim& 
(2^kr_1r_2)^{-1/2}.
\end{align}
Indeed, this is from \eqref{tildeK} dropping off the cutoff function $\mathbbm{1}_{[0,\pi]}(|\theta|) $ by using $\int_{\pi-2\epsilon}^{\pi}$, and we use \eqref{equ:Gjktermdef} with $j=0$ to see
 $$d(r_1,r_2; \theta)=|x-y|=\sqrt{r_1^2+r_2^2-2r_1r_2\cos(\theta)}.$$

Now we aim to show \eqref{equ:kernconfrou}.
Due to the compact support of $\eta$,  we note that
$\theta\to \pi$ as $\epsilon\to 0$. Hence the facts $|x-y|\sim 1$ (on the support of $\psi$) implies $r_1+r_2\sim 1$.
Hence, we can use the Van der Corput lemma \ref{lem:VCL} to obtain \eqref{equ:kernconfrou} if we could verify 
\begin{equation}\label{est:2-d}
|\partial_{\theta}^2d(r_1,r_2; \theta_1-\theta_2)|\geq c r_1r_2.
\end{equation}
Actually, as $\epsilon\to 0$,  we compute that
\begin{equation}\label{1-d}
\begin{split}
\partial_{\theta}d(r_1,r_2; \theta)&=\frac{r_1r_2\sin(\theta)}{\sqrt{r_1^2+r_2^2-2r_1r_2\cos(\theta)}}\\
&=\frac{r_1r_2}{r_1+r_2}\sin(\theta)+O((r_1r_2)^2(\theta-\pi)^3),
\end{split}
\end{equation}
and
\begin{equation}\label{2-d}
\begin{split}
&\partial^2_{\theta}d(r_1,r_2; \theta)=\frac{r_1r_2}{r_1+r_2}\cos(\theta)+O((r_1r_2)^2(\theta-\pi)^2),\end{split}
\end{equation}
which implies \eqref{est:2-d}. Thus, we have proved \eqref{2-2}. By using \eqref{2-2} and the facts that $r_1+r_2\sim 1$ and $\theta_2\in[0, \epsilon)$, we use the inequality \eqref{q-q0} to
show
\begin{equation*}
\begin{split}
\big\|\tilde{T}_{G_{m,k}} f\big\|_{L^2(\R^2)} &\leq  C 2^{-k(\frac32+\delta)} 2^{-\frac k2} \int_0^1 r_2^{-\frac12}\|f(r_2,\theta_2)\|_{L^2_{\theta_2}(\R)} \,r_2 dr_2 \\
& \leq  C 2^{-k(\frac32+\delta)} 2^{-\frac k2} \|f\|_{L^{2}(\R^2)},
\end{split}
\end{equation*}
which implies \eqref{est:TGj'2}.\vspace{0.2cm}

We next prove \eqref{est:TGj'4} by considering \eqref{q-q0} with $p>4$. Note that $p>4$, then set $p_0:=\frac p2>2$ which allows us to use the bilinear argument. 
To this end, we write that
\begin{equation*}
\begin{split}
&\Big\| [\tilde{T}_{G_{m,k},r_1,r_2} f(r_2, \cdot)]( \theta_1) \Big\|^2_{L^p(\R^2)} \\&
=\Big\|  [\tilde{T}_{G_{m,k},r_1,r_2} f(r_2, \cdot)]( \theta_1)  \overline{[\tilde{T}_{G_{m,k},r_1,r_2} f(r_2, \cdot)]( \theta_1) } \Big\|_{L^{p_0}(\R^2)}.
\end{split}
\end{equation*}
Let $F(r_2, \theta_2; \theta'_2)=f(r_2, \theta_2) \overline{f(r_2, \theta'_2)}$, we define 
\begin{equation}\label{bfTr1r22}
\begin{split}
&\big({\bf T}_{G_{m,k},r_1,r_2} F\big)(\theta_1)\\
&= \int_0^{2\pi} \int_0^{2\pi} {\bf K}_{G_{m,k},r_1,r_2} (\theta_1,\theta_2,\theta'_2) F(r_2, \theta_2; \theta'_2) d\theta_2\,  d\theta'_2, 
\end{split}
\end{equation}
where letting $x=r_1(\cos\theta_1,\sin\theta_1), y=r_2(\cos\theta_2,\sin\theta_2),  z=r_2(\cos\theta'_2,\sin\theta'_2)$, the kernel 
\begin{equation}\label{bfTr1r2}
\begin{split}
&{\bf K}_{G_{m,k},r_1,r_2} (\theta_1,\theta_2,\theta'_2)\\
&= \tilde{K}_{G_{m,k}}(r_1,r_2;\theta_1-\theta_2)  \overline{ \tilde{K}_{G_{m,k}}(r_1,r_2;\theta_1-\theta'_2)} \eta(\theta_1-\theta_2) \eta(\theta_1-\theta'_2) \\
&=2^{-2k(\frac32+\delta)} b(2^k, |x-y|) b(2^k, |x-z|)e^{i2^k (|x-y|-|x-z|)}\\
&\qquad\times\mathbbm{1}_{[0,\pi]}(|\theta_1-\theta_2|)\eta(\theta_1-\theta_2) \mathbbm{1}_{[0,\pi]}(|\theta_1-\theta'_2|) \eta(\theta_1-\theta'_2)
\end{split}
\end{equation}
with $b$ satisfying \eqref{est:bk}.\vspace{0.2cm}

For our purpose, we need the following analogue of \cite[Lemma 4.1.]{MYZ}, which can be proved in the same way as there and we refer readers to \cite[Lemma 4.1]{MYZ} for the detailed proof. 
\begin{lemma}\label{lem:key1} Let ${\bf T}_{G_{m,k},r_1,r_2}$ be the operator defined in \eqref{bfTr1r22}. Then
\begin{equation}\label{est:inf1}
\begin{split}
&\Big\|\big({\bf T}_{G_{m,k},r_1,r_2} F(r_2, \theta_2; \theta'_2)\big)(\theta_1)\Big\|_{L^\infty (r_1 dr_1 d\theta_1)}
\\&\qquad\lesssim 2^{-2k(\frac32+\delta)} \|F\|_{L^{1}( d \theta_2 d\theta'_2)},
\end{split}
\end{equation}
and
\begin{equation}\label{est:22}
\begin{split}
&\Big\|\big({\bf T}_{G_{m,k},r_1,r_2} F(r_2, \theta_2; \theta'_2)\big)(\theta_1)\Big\|_{L^2 (r_1 dr_1 d\theta_1)}\\
&\qquad\lesssim 2^{-2k(\frac32+\delta)}
(1+2^kr^3_2)^{-\frac12} \|F\|_{L^{2}( d \theta_2 d\theta'_2)}.
\end{split}
\end{equation}
\end{lemma}

\medskip
By interpolation with \eqref{est:inf1} and \eqref{est:22}, we obtain
\begin{equation*}
\begin{split}
&\Big\|\big({\bf T}_{G_{m,k},r_1,r_2} F(r_2, \theta_2; \theta'_2)\big)(\theta_1)\Big\|_{L^{p_0} (r_1 dr_1 d\theta_1)}\\
&\lesssim 2^{-2k(\frac32+\delta)} (1+2^kr^3_2)^{-\frac1{p_0}}\|F\|_{L^{p_0'}( d \theta_2 d\theta'_2)}\lesssim 2^{-2k(\frac32+\delta)} (1+2^kr^3_2)^{-\frac1{p_0}}\|f\|^2_{L^{p_0'}( d \theta_2)}.
\end{split}
\end{equation*}
Therefore, from \eqref{q-q0} again,  since $p_0'< 2<p$, we use the H\"older inequality to prove
\begin{equation}
\begin{split}
\big\|\tilde{T}_{G_{m,k}} f\big\|_{L^p(\R^2)} &\lesssim 2^{-k(\frac32+\delta)}  \int_0^1 (1+2^kr^3_2)^{-\frac1{2p_0}}\|f\|_{L^{p_0'}( d \theta_2)} \,r_2 dr_2\\
&\lesssim 2^{-k(\frac32+\delta)}  \Big(\int_0^1 (1+2^kr^3_2)^{-\frac{p'}{p}}\,r_2 dr_2\Big)^{\frac1{p'}} \|f\|_{L^p_{r_2dr_2}L^{p_0'}_{d \theta_2}([0, \epsilon))} \\
&\lesssim 2^{-k(\frac32+\delta)}  2^{-\frac{2k}{3p'}}\|f\|_{L^p(\R^2)}\lesssim 2^{-k(\frac32+\delta)}  2^{-\frac{2k}{p}}\|f\|_{L^p(\R^2)},
\end{split}
\end{equation}
where we observe that $2^{-\frac{2k}{3p'}}\leq 2^{-\frac{2k}{p}}$ for $p>4$ and $k\geq 0$.
Therefore, we finally prove \eqref{est:TGj'4}.
\end{proof}

Therefore, we have obtained \eqref{equ:TG1kest} and completed the proof of Proposition \ref{prop:Tg1kest}.
\end{proof}

With Proposition \ref{prop:Tg1kest} in hand, we get
\begin{align*}
\|T_{G_m^\pm}f\|_{L^p(X)}\leq&\sup_{0\leq j\leq C/\delta}\|T_{G_{mj}^\pm}f\|_{L^p(X)}\\
\leq&\sup_{0\leq j\leq C/\delta}\sum_{k\geq0}\|T_{G_{mj,k}^\pm}f\|_{L^p(X)}\leq C\|f\|_{L^p(X)}
\end{align*}
for $p=2$, $p>4$ and  $\delta<\delta_c(p,2)$. By interpolation and duality, we obtain \eqref{equ:TG1pmterm} for $p\geq1$ and $\delta<\delta_c(p,2)$.

\subsection{The estimate of $T_{D_m}^\pm$} Compared with the geometric term $T_{G_m^\pm}$ estimated in last subsection, 
the diffractive term $T_{D_m}^\pm$ has no jump function but has singularity in $A_\sigma(s,\theta_1,\theta_2)$, especially in the case that $s\to 0$ and $(\theta_1-\theta_2)\to \pm \pi$. 
So we have to be careful to treat the diffractive term $T_{D_m}^\pm$.

Now, we turn to prove the estimate of the term $T_{D_m}^\pm$, that is
\begin{equation}\label{equ:TD1pmest}
\|T_{D_m}^\pm f\|_{L^p(X)}\leq C\|f\|_{L^p(X)}.
\end{equation}
Recall \eqref{equ:Dktermdef} and \eqref{equ:dsdefr12}
$$D_{m}^\pm(\delta;r_1,\theta_1;r_2,\theta_2)=-\frac1{\pi \sigma}\int_0^\infty \frac{a_\pm(d_s)e^{\pm id_s}}{(1+d_s)^{\frac{3}2+\delta}}A_\sigma(s,\theta_1,\theta_2)\;ds,$$
with
\begin{align*}
d_s=&d_s(r_1,r_2,\theta_1,\theta_2)=\sqrt{r_1^2+r_2^2+2r_1r_2\cosh(s)},
\end{align*}
and recall \eqref{equ:Asigmadef}
\begin{align*}
A_\sigma(s,\theta_1,\theta_2)=&\frac12\frac{\sin\big(\tfrac{\pi-(\theta_1-\theta_2)}{\sigma}\big)}{\cosh\frac{s}{\sigma}-\cos\big(\tfrac{\pi-(\theta_1-\theta_2)}{\sigma}\big)}
+\frac12\frac{\sin\big(\tfrac{\pi+(\theta_1-\theta_2)}{\sigma}\big)}{\cosh\frac{s}{\sigma}-\cos\big(\tfrac{\pi+(\theta_1-\theta_2)}{\sigma}\big)}.
\end{align*}
Since the case with $D_{m}^-$ can be shown by the same argument as the $+$ case, we only consider the case with $D_{m}^+$. 
More precisely, it suffices to show
\begin{equation}\label{equ:TD1+est}
\begin{split}
\big\|\int_0^\infty\int_0^{2\sigma\pi} K_{D_m^+}(r_1,r_2;\theta_1-\theta_2) &f(r_2,\theta_2)d\theta_2\;r_2dr_2\big\|_{L^p_{\theta_1}([0,2\sigma\pi],L^p(r_1dr_1))}\\
&\leq C\|f\|_{L^p(X)},
\end{split}
\end{equation}
with
\begin{equation}\label{equ:D1+def}
K_{D_m^+}(r_1,r_2;\theta)=\int_0^\infty \frac{a_+(d_s)e^{id_s}}{(1+d_s)^{\frac{3}2+\delta}}
\frac{\sin\big(\tfrac{\pi\pm \theta}{\sigma}\big)}{\cosh\frac{s}{\sigma}-\cos\big(\tfrac{(\pi\pm \theta)}{\sigma}\big)}
\;ds,
\end{equation}
and $a_+$ is bounded from below near infinity and satisfies  \eqref{eq:b,a-symbol-bounds}.
By partition of unity
\eqref{equ:paruni}, we decompose
\begin{equation}\label{equ:KD1+k}
K_{D_m^+}(r_1,r_2;\theta_1-\theta_2)=\sum_{k\geq0}K_{D_{m,k}^+}(r_1,r_2;\theta_1-\theta_2),
\end{equation}
with
\begin{equation}\label{equ:kD1k+term}
\begin{split}
&K_{D_{m,k}^+}(r_1,r_2;\theta_1-\theta_2)\\
&=K_{D_{m}^+}(r_1,r_2;\theta_1-\theta_2)\times
\begin{cases}
\psi(2^{-k}(r_1+r_2)),\quad &\text{if}\;k\geq1,\\
\psi_0(r_1+r_2),\quad &\text{if}\;k=0.
\end{cases}
\end{split}
\end{equation}

Then, we can prove \eqref{equ:TD1+est} if we could prove the localized estimates:
\begin{proposition}\label{prop:keyTD1+k}
Let $T_{D_{m,k}^+}$ be the operator associated with the kernel $K_{D_{m,k}^+}(r_1,r_2;\theta_1-\theta_2)$. Then, there holds
\begin{equation}\label{equ:TD1+kesti}
\|T_{D_{m,k}^+}\|_{L^p(X)\to L^p(X)}\leq C\times 
\begin{cases}
2^{-k\delta},\quad &\text{if}\quad p=2,\\
2^{k(\delta_c(p,2)-\delta)},\quad&\text{if}\quad p>4,
\end{cases}
\end{equation}
for any $k\geq0$.

\end{proposition}

\begin{remark} In particular, for $p=4$, we can obtain
\begin{equation} 
\|T_{D_{m,k}^+}\|_{L^4(X)\to L^4(X)}\leq C2^{-k(\delta-\epsilon)},
\end{equation}
where $0<\epsilon\ll 1$. This is enough to show \eqref{equ:TD1pmest} since if we choose $0<\epsilon\leq \frac\delta 2$.
\end{remark}

Indeed, by Proposition \ref{prop:keyTD1+k}, we have
\begin{align*}
\|T_{D_m^+}f\|_{L^p(X)}\leq&\sum_{k\geq0}\|T_{D_{m,k}^+}f\|_{L^p(X)}\leq C\|f\|_{L^p(X)}
\end{align*}
for $p=2$ and $p>4$ with $\delta>\delta_c(p,2)$. By interpolation and duality, we obtain \eqref{equ:TD1pmest}
for $p\geq1$ and $\delta>\delta_c(p,2)$.

Therefore, we conclude the proof of Theorem \ref{thm:main}
if we could prove that Proposition \ref{prop:keyTD1+k} holds true, which is the main task of the next section.

\section{Proof of Proposition \ref{prop:keyTD1+k}}\label{sec:TD1+kterm}

In this section, we aim to show \eqref{equ:TD1+kesti} in Proposition \ref{prop:keyTD1+k}, which we recall here:
\begin{equation}\label{equ:TD1+kesti-re}
\|T_{D_{m,k}^+}\|_{L^p(X)\to L^p(X)}\leq C\times 
\begin{cases}
2^{-k\delta},\quad &\text{if}\quad p=2,\\
2^{k(\delta_c(p,2)-\delta)},\quad &\text{if}\quad p>4,
\end{cases}
\end{equation}
with its kernel defined in \eqref{equ:kD1k+term} and \eqref{equ:D1+def}.
We first consider the case $k=0$:
\begin{equation}\label{equ:TD1k+k0}
\|T_{D_{m, 0}^+}\|_{L^p(X)\to L^p(X)}\leq C,\quad \forall\;p\geq1,
\end{equation}
which implies \eqref{equ:TD1+kesti-re} with $k=0$.
Similarly to the proof of \eqref{equ:Asigmaintbd}, there exists a constant $C$ independent of $\theta$ such that
\begin{align}\label{est:A-sigma}
\int_0^\infty \left| \frac{\sin(\theta)}{\cosh\frac{s}{\sigma}-\cos(\theta)}\right|\;ds\leq C<+\infty, 
\end{align}
where
\begin{align*}
\theta=\frac1\sigma(\pi\pm (\theta_1-\theta_2)).
\end{align*}
We prove this by decomposing the integral into the part with $s<\sigma$ and the part with $s \geq \sigma$.
For the part with $s \geq \sigma$, we have
\begin{equation*}
\begin{split}
 \int_\sigma^\infty \left| \frac{\sin(\theta)}{\cosh\frac{s}{\sigma}-\cos(\theta)}\right| ds\leq  \sigma \int_1^\infty e^{-s/2} ds\leq C_\sigma.
\end{split}
\end{equation*}
For the part with $s < \sigma$, we have
\begin{equation*}
\begin{split}
 \int_0^\sigma \left| \frac{\sin(\theta)}{\cosh\frac{s}{\sigma}-\cos(\theta)}\right|  ds&\leq  \frac{\sigma}2 \int_0^1 \frac{|\sin\theta| }{\sinh^2(\frac s2)+\sin^2(\frac\theta2)} ds\\
 &\leq C_\sigma \int_0^3 \frac{ |\sin(\frac{\theta}2)| }{s^2+\sin^2(\frac\theta2)} ds\leq  C_\sigma \int_0^\infty \frac{1}{s^2+1} ds\leq C.
\end{split}
\end{equation*}
This together with the fact that $d_s\geq r_1+r_2$ yields that
$$\big|K_{D_{m,0}^+}(r_1,r_2;\theta_1-\theta_2)\big|\leq C\psi_0(r_1+r_2).$$
Combining this with Lemma \ref{lem:logYoung}, we obtain
\begin{align*}
\|T_{D_{m,0}^+}f(x)\|_{L^p(X)}
\leq&C\left\|\int_0^\infty\int_0^{2\sigma\pi}\psi_0(r_1+r_2)\big|f(r_2,\theta_2)\big|d\theta_2\;r_2dr_2\right\|_{L^p_{\theta_1}([0,2\sigma\pi],L^p(r_1dr_1))}\\
\leq&C\left\|\int_0^\infty\int_0^{2\sigma\pi}\frac{\big|f(r_2,\theta_2)\big|}{(1+r_1+r_2)^3}d\theta_2\;r_2dr_2\right\|_{L^p_{\theta_1}([0,2\sigma\pi],L^p(r_1dr_1))}\\
\leq&C\|f\|_{L^p(X)},
\end{align*}
which implies \eqref{equ:TD1k+k0}. 

Next we consider the case in which $k\geq1$. We define
\begin{equation}\label{def:tilde-K}
\begin{split}
\tilde{K}_{D_{m,k}^+}(r_1,r_2,\theta_1-\theta_2):&=K_{D_{m,k}^+}(2^kr_1,2^kr_2,\theta_1-\theta_2)\\
&
=\psi(r_1+r_2)K_{D_m^+}(2^kr_1,2^kr_2,\theta_1-\theta_2),
\end{split}
\end{equation}
where the last equality follows from the definition in \eqref{equ:kD1k+term}.
Its associated operator $\widetilde{T}_{D_{m,k}^+}$ acts by
\begin{equation}\label{equ:TscalD1+}
\widetilde{T}_{D_{m,k}^+}f(r_1,\theta_1)=\int_0^\infty\int_0^{2\sigma\pi}\tilde{K}_{D_{m,k}^+}(r_1,r_2,\theta_1-\theta_2)f(r_2,\theta_2)\;d\theta_2\;r_2dr_2.
\end{equation}
After rescaling, we have  
\begin{align*}
\widetilde{T}_{D_{m,k}^+}f(r_1,\theta_1)=&\int_0^\infty\int_0^{2\sigma\pi}K_{D_{m,0}^+}(2^kr_1,2^kr_2,\theta_1-\theta_2)f(r_2,\theta_2)\;d\theta_2\;r_2dr_2\\
=&2^{-2k}\int_0^\infty\int_0^{2\sigma\pi}K_{D_{m,0}^+}(2^kr_1,r_2,\theta_1-\theta_2)f(2^{-k}r_2,\theta_2)\;d\theta_2\;r_2dr_2\\
=&2^{-2k}T_{D_{m,k}^+}\big(f(2^{-k}r_2,\theta_2)\big)(2^kr_1,\theta_1)
\end{align*}
which is also equivalent to
\begin{equation}\label{equ:scTDk}
T_{D_{m,k}^+}f(r_1,\theta_1)=2^{2k}\widetilde{T}_{D_{m,k}^+}\big(f(2^kr_2,\theta_2)\big)(2^{-k}r_1,\theta_1).
\end{equation}
Hence, our goal \eqref{equ:TD1+kesti-re}  is equivalent to the following proposition:
\begin{proposition}\label{prop:scalprop}
There exists a constant $C$ such that
\begin{equation}\label{equ:wideTD1+kesti}
\|\widetilde{T}_{D_{m,k}^+}\|_{L^p(X)\to L^p(X)}\leq C2^{-2k}\times 
\begin{cases}
2^{-k\delta},\quad &\text{if}\quad p=2,\\
2^{k(\delta_c(p,2)-\delta)},\quad &\text{if}\quad p>4.
\end{cases}
\end{equation}

\end{proposition}

\begin{proof}[The proof Proposition \ref{prop:scalprop}] We split the operator $\widetilde{T}_{D_{m,k}^+}$ into two operators by dividing the kernel into two pieces due to the singularities as $s\to0$.

Recall the definition \eqref{def:tilde-K}, which in turn used \eqref{equ:D1+def}, we decompose 
\begin{align*}
\tilde{K}_{D_{m,k}^+}(r_1,r_2,\theta_1-\theta_2)
\triangleq\tilde{K}_{D_{m,k}^+}^1(r_1,r_2,\theta_1-\theta_2)+\tilde{K}_{D_{m,k}^+}^2(r_1,r_2,\theta_1-\theta_2),
\end{align*}
where
\begin{equation}\label{equ:KD11k}
\begin{split}
&\tilde{K}_{D_{m,k}^+}^1(r_1,r_2,\theta_1-\theta_2)\\
=&2^{-k(\frac32+\delta)}\psi(r_1+r_2)\int_{2\sigma}^\infty e^{i2^kd_s} \frac{a_+(2^kd_s)}{(2^{-k}+d_s)^{\frac{3}2+\delta}}\frac{\sin\big(\tfrac{\pi\pm(\theta_1-\theta_2)}{\sigma}\big)}{\cosh\frac{s}{\sigma}-\cos\big(\tfrac{(\pi\pm(\theta_1-\theta_2)}{\sigma}\big)}\;ds
\end{split}
\end{equation}
and
\begin{equation}
\label{equ:KD12k}
\begin{split}
&\tilde{K}_{D_{m,k}^+}^2(r_1,r_2,\theta_1-\theta_2)\\
=&2^{-k(\frac32+\delta)}\psi(r_1+r_2)\int_0^{2\sigma} e^{i2^kd_s}\frac{a_+(2^kd_s)}{(2^{-k}+d_s)^{\frac{3}2+\delta}}\frac{\sin\big(\tfrac{\pi\pm(\theta_1-\theta_2)}{\sigma}\big)}{\cosh\frac{s}{\sigma}-\cos\big(\tfrac{(\pi\pm(\theta_1-\theta_2)}{\sigma}\big)}\;ds,
\end{split}
\end{equation}
with $d_s=\sqrt{r_1^2+r_2^2+2r_1r_2\cosh(s)}$. Hence the operator 
$$\widetilde{T}_{D_{m,k}^+}=\widetilde{T}_{D_{m,k}^+}^1+\widetilde{T}_{D_{m,k}^+}^2$$
where $\widetilde{T}_{D_{m,k}^+}^j(j=1,2)$ is the operator associated with the kernel $\tilde{K}_{D_{m,k}^+}^j(j=1,2)$, respectively. 

%MARK
Now we first estimate the term $\widetilde{T}_{D_{m,k}^+}^1$. To this end, we need a lemma about the kernel's bound. 
\begin{lemma}\label{lem:tildeKD1+}
Let $\tilde{K}_{D_{m,k}^+}^1(r_1,r_2,\theta_1-\theta_2)$ be given in \eqref{equ:KD11k}, then there exists a constant $C$ such that
\begin{equation}\label{equ:tildeKD1+}
\big|\tilde{K}_{D_{m,k}^+}^1(r_1,r_2,\theta_1-\theta_2)\big|\leq C2^{-k(\frac32+\delta)}
(1+2^kr_1r_2)^{-\frac12}\psi(r_1+r_2).
\end{equation}

\end{lemma}
Suppose this lemma has been proved, we apply
 Proposition \ref{prop:TD1+key} ( with $p=2$ and $p>4$, respectively) to obtain
\begin{equation}\label{equ:widTK1fest}
\begin{split}
\|\widetilde{T}_{D_{m,k}^+}^1\|_{L^p(X)\to L^p(X)}\lesssim& 2^{-k(\frac32+\delta)}\times
\begin{cases}
2^{-\frac{k}2}\quad&\text{if}\quad p=2,\\
2^{-\frac2pk}\quad &\text{if}\quad p>4,
\end{cases}\\
\lesssim&2^{-2k}\times 
\begin{cases}
2^{-k\delta},\quad &\text{if}\quad p=2,\\
2^{k(\delta_c(p,2)-\delta)},\quad&\text{if}\quad p>4,
\end{cases}
\end{split}
\end{equation}
where we have used $\delta_c(p,2)=\max\big\{2\big|\tfrac12-\tfrac1p\big|-\tfrac12,0\big\}=\frac12-\frac2p$ for $p>4$.

\begin{proof}[The proof of Lemma \ref{lem:tildeKD1+}]
By \eqref{equ:KD11k}, it suffices to prove
\begin{equation}\label{equ:TD1kredfk}
\left| \psi(r_1+r_2) \int_{2\sigma}^\infty e^{i2^kd_s} \frac{a_+(2^kd_s)}{(2^{-k}+d_s)^{\frac{3}2+\delta}}\frac{\sin\big(\tfrac{\pi\pm(\theta_1-\theta_2)}{\sigma}\big)}{\cosh\frac{s}{\sigma}-\cos\big(\tfrac{(\pi\pm(\theta_1-\theta_2)}{\sigma}\big)}\;ds\right|\leq C(1+2^kr_1r_2)^{-\frac12}.
\end{equation}
Due to the support of $\psi$, we only need to concern the case $r_1+r_2 \simeq 1$, which implies $d_s\geq r_1+r_2\gtrsim 1$. 

{\bf Case 1: $2^kr_1r_2\lesssim 1.$} From \eqref{eq:b,a-symbol-bounds}, it follows
$$|a_+(2^kd_s)|\leq C,$$ hence, combining with \eqref{est:A-sigma}, we obtain
\begin{align*}
\text{LHS of}\; \eqref{equ:TD1kredfk}
\lesssim&\int_0^\infty\left|\frac{\sin\big(\tfrac{\pi\pm(\theta_1-\theta_2)}{\sigma}\big)}{\cosh\frac{s}{\sigma}-\cos\big(\tfrac{(\pi\pm(\theta_1-\theta_2)}{\sigma}\big)}\right|\;ds\leq C,
\end{align*}
which implies  \eqref{equ:TD1kredfk} when $2^kr_1r_2\lesssim 1$.

{\bf Case 2: $2^kr_1r_2\gg 1.$} By a direct computation, we obtain
\begin{align*}
\pa_sd_s=&\frac{r_1r_2\sinh(s)}{\sqrt{r_1^2+r_2^2+2r_1r_2\cosh(s)}},\\
\pa_s^2d_s=&\frac{r_1r_2\cosh(s)}{\sqrt{r_1^2+r_2^2+2r_1r_2\cosh(s)}}-\frac{(r_1r_2\sinh(s))^2}{(r_1^2+r_2^2+2r_1r_2\cosh(s))^\frac32}>0,
\end{align*}
where the last inequality is equivalent to 
\begin{align*}
r_1r_2\cosh(s)(r_1^2+r_2^2+2r_1r_2\cosh(s)) > (r_1r_2\sinh(s))^2,
\end{align*}
which follows by $\cosh s > \sinh s$, which shows $\pa_sd_s$ is monotonic.

In the region $s \geq 2\sigma>2$ and on the support of $\psi(r_1+r_2)$, we know that 
\begin{align*}
\pa_sd_s\gtrsim \min((r_1r_2)^\frac{1}{2},r_1r_2). 
\end{align*}
This follows by considering two cases
\begin{enumerate}
\item $e^sr_1r_2 \gg r_1^2+r_2^2$;
\item $e^sr_1r_2 \lesssim r_2^2+r_2^2$.
\end{enumerate} 
In the first case, we have 
\begin{align*}
\pa_sd_s = \frac{r_1r_2\sinh(s)}{\sqrt{r_1^2+r_2^2+2r_1r_2\cosh(s)}} \gtrsim \frac{r_1r_2e^s}{(r_1r_2e^s)^{1/2}} 
\gtrsim (r_1r_2)^{1/2}.
\end{align*}
In the second case, we have
\begin{align*}
\pa_sd_s = \frac{r_1r_2\sinh(s)}{\sqrt{r_1^2+r_2^2+2r_1r_2\cosh(s)}}
\gtrsim r_1r_2\sinh(s) \gtrsim r_1r_2,
\end{align*}
where we used $r_1,r_2 \lesssim 1$ on the support of $\psi(r_1+r_2)$.

Let 
\begin{equation}\label{def:F-k}
F_k(s,d_s;\theta_1-\theta_2)=\frac{a_+(2^kd_s)}{(2^{-k}+d_s)^{\frac{3}2+\delta}}\frac{\sin\big(\tfrac{\pi\pm(\theta_1-\theta_2)}{\sigma}\big)}{\cosh\frac{s}{\sigma}-\cos\big(\tfrac{(\pi\pm(\theta_1-\theta_2)}{\sigma}\big)}.
\end{equation}
Hence, by Van der Corput Lemma \ref{lem:VCL}, we obtain
\begin{align}\nonumber
\text{LHS of}\; \eqref{equ:TD1kredfk}\lesssim&\left| \int_{2\sigma}^\infty e^{i2^kd_s}F_k(s,d_s;\theta_1-\theta_2)\;ds\right|\\\nonumber
\lesssim&\big(2^k(r_1r_2)^\frac12\big)^{-1}\left(\sup_{s\geq2\sigma}\big|F_k (s,d_s;\theta_1-\theta_2)\big|
+\int_{2\sigma}^\infty\left|\frac{\pa}{\pa s}F_k (s,d_s;\theta_1-\theta_2)\;ds\right|\right)\\\label{equ:Regio2sigma}
\lesssim&(1+2^kr_1r_2)^{-\frac12}.
\end{align}

\end{proof}

Next we prove \eqref{equ:wideTD1+kesti} for $\widetilde{T}_{D_{m,k}^+}^2$. If $2^kr_1r_2\lesssim1$, as arguing Case 1 in Lemma \ref{lem:tildeKD1+}, we also obtain 
\begin{equation*}
\big|\tilde{K}_{D_{m,k}^+}^2(r_1,r_2,\theta_1-\theta_2)\big|\leq C2^{-k(\frac32+\delta)}
(1+2^kr_1r_2)^{-\frac12}\psi(r_1+r_2).
\end{equation*}
Hence, as before, we have \eqref{equ:wideTD1+kesti}. So it remains to estimate 
\begin{equation}\label{equ:2kr1r2gg1}
\left\|\int_0^{\infty}\int_0^{2\sigma\pi}\chi_1(2^kr_1r_2)\tilde{K}_{D_{m,k}^+}^2(r_1,r_2,\theta_1-\theta_2)f(r_2,\theta_2)\;d\theta_2\;r_2dr_2 \right\|_{L^p_{\theta_1}([0,2\sigma\pi],L^p(r_1dr_1))},
\end{equation}
where $\chi_1$ is a smooth function supported on $[C,\infty)$ and is identically $1$ on $[2C,\infty)$ with some large constant $C$.
%we are restricted in the region $2^kr_1r_2\gg1$. 
In contrast to \eqref{equ:widTK1fest}, here we are considering $0\leq s\leq2\sigma$. If we do as before, we still have
\begin{equation}\label{equ:pasnocrit}
\pa_sd_s\big|_{s=0}=0,\quad \big|\pa_s^2d_s\big|\geq c\frac{r_1r_2}{r_1+r_2}\gtrsim r_1r_2,
\end{equation}
but if the derivative hits the factor
$$\frac{\sin\big(\tfrac{\pi\pm(\theta_1-\theta_2)}{\sigma}\big)}{\cosh\frac{s}{\sigma}-\cos\big(\tfrac{(\pi\pm(\theta_1-\theta_2)}{\sigma}\big)},$$
it can becomes singular as $s \to 0$ when $\cos\big(\tfrac{(\pi\pm(\theta_1-\theta_2)}{\sigma}\big) = 1$.
Consequently, we can not obtain
$$\int_0^{2\sigma}\left|\frac{\pa}{\pa s}F_k (s,d_s;\theta_1-\theta_2)\;ds\right| \;ds\lesssim1,$$
which is needed for applying the Van der Corput Lemma, as we did in Case 2 of Lemma \ref{lem:tildeKD1+}.

To overcome this, we further decompose the integral into three pieces 
\begin{align*}
&\int_0^{2\sigma} e^{i2^kd_s}F_k(s,d_s;\theta)\;ds\\
=&\int_0^{2\sigma} e^{i2^kd_s}\frac{a_+(2^kd_s)}{(1+2^kd_s)^{\frac{3}2+\delta}}
\left(\frac{\sin\big(\tfrac{\pi\pm \theta}{\sigma}\big)}{\cosh\frac{s}{\sigma}-\cos\big(\tfrac{(\pi\pm\theta}{\sigma}\big)}
-\frac{4\sigma(\pi\pm\theta)}{s^2+(\pi\pm\theta)^2}\right)
\;ds\\
&+\int_0^{2\sigma} e^{i2^kd_s}\left(\frac{a_+(2^kd_s)}{(1+2^kd_s)^{\frac{3}2+\delta}}-\frac{a_+(2^k(r_1+r_2))}{(1+2^k(r_1+r_2))^{\frac32+\delta}}\right)\frac{4\sigma(\pi\pm\theta)}{s^2+(\pi\pm\theta)^2}\;ds\\
&+\int_0^{2\sigma} e^{i2^kd_s}\frac{a_+(2^k(r_1+r_2))}{(1+2^k(r_1+r_2))^{\frac32+\delta}}\frac{4\sigma(\pi\pm \theta)}{s^2+(\pi\pm \theta)^2}\;ds\\
\triangleq&E_1(r_1,r_2;\theta)+E_2(r_1,r_2;\theta)+E_3(r_1,r_2;\theta).
\end{align*}
So, for $j=1,2,3$, we turn to estimate the terms
\begin{equation}\label{equ:TGjterm}
T_{K_{k,j}}f:=\int_0^{\infty}\int_0^{2\sigma\pi}K_{k,j}(r_1,r_2; \theta_1-\theta_2)f(r_2,\theta_2)\;d\theta_2\;r_2dr_2,
\end{equation}
where the kernels $K_{k,j}$ ($k\geq1$ and $j=1,2,3$) are given by 
\begin{equation}\label{def:Kkj}
K_{k,j}(r_1,r_2; \theta)= \chi_1(2^kr_1r_2) \psi(r_1+r_2)E_j(r_1,r_2,\theta).
\end{equation}
We have the following estimate for these three kernels. 
%satisfy the following Proposition.
\begin{proposition}\label{prop:TEjest}
Let $K_{k,j}(r_1,r_2; \theta)$ be given in \eqref{def:Kkj}  for $k\geq1$ and $j=1,2, 3$. Then there exists a constant $C$ independent of $\theta$ such that: 
\begin{itemize}
 \item for $j=1,2$
\begin{equation}\label{equ:TE12est}
\big| K_{k,j}(r_1,r_2; \theta)\big|\leq C(1+2^kr_1r_2)^{-\frac12}\psi(r_1+r_2);
\end{equation}
\item and for $j=3$, 
\begin{equation}\label{equ:TE3est}
\begin{split}
\Big| e^{-i2^k(r_1+r_2)}K_{k,3}(r_1,r_2; \theta)-&M_k(r_1,r_2; \theta)\Big|\\
&\leq C(1+2^kr_1r_2)^{-\frac12}\psi(r_1+r_2),
\end{split}
\end{equation}
where
\begin{equation}\label{equ:Mrthedef}
\begin{split}
M_k(r_1,r_2; \theta)=\psi(r_1+r_2)&\chi_1(2^kr_1r_2) \frac{a_+(2^k(r_1+r_2))}{(1+2^k(r_1+r_2))^{\frac32+\delta}}\\
&\quad\times\int_0^{2\sigma} e^{i2^k\frac{r_1r_2}{r_1+r_2}s^2}\frac{4\sigma(\pi\pm\theta)}{s^2+(\pi\pm\theta)^2}\;ds.
\end{split}
\end{equation}
\end{itemize}
\end{proposition}

\begin{proof}
We first prove \eqref{equ:TE12est}.
Recalling \eqref{equ:pasnocrit},  we use the Van der Corput Lemma \ref{lem:VCL} again
to obtain
\begin{equation}\label{equ:E1E2est}
|E_1(r_1,r_2;\theta)|+|E_2(r_1,r_2;\theta)|\leq C(2^kr_1r_2)^{-\frac12}\lesssim (1+2^kr_1r_2)^{-\frac12},
\end{equation}
if we can verify 
\begin{align*}
\int_0^{2\sigma}\left|\frac{\pa}{\pa s}\left[\frac{a_+(2^kd_s)}{(1+2^kd_s)^{\frac{3}2+\delta}}
\left(\frac{\sin\big(\tfrac{\theta}{\sigma}\big)}{\cosh\frac{s}{\sigma}-\cos\big(\tfrac{\theta}{\sigma}\big)}
-\frac{4\sigma\theta}{s^2+\theta^2}\right) \right]\right|\;ds\leq& C,\\
\int_0^{2\sigma}\left|\frac{\pa}{\pa s}\left[\left(\frac{a_+(2^kd_s)}{(1+2^kd_s)^{\frac{3}2+\delta}}-\frac{a_+(2^k(r_1+r_2))}{(1+2^k(r_1+r_2))^{\frac32+\delta}}\right)\frac{4\sigma \theta}{s^2+\theta^2} \right]\right|\;ds\leq& C,
\end{align*}
where $C$ is a finite constant independent of $\theta$.
Indeed, this can be verified by the same argument as in \cite[(4.48),(4.49)]{WX}.

Next we prove \eqref{equ:TE3est}.
By modifying \cite[(4.52)]{WX} and by replacing $(r_1+r_2)^{-\frac12}$ by $(1+2^k(r_1+r_2))^{-(\frac32+\delta)}$,
we get
$$\Big| e^{-i2^k(r_1+r_2)} K_{k,3}(r_1,r_2; \theta)-M_k(r_1,r_2; \theta)\Big|\leq C(1+2^kr_1r_2)^{-\frac12},$$
which completes the proof of Proposition \ref{prop:TEjest}.

\end{proof}

We use Proposition \ref{prop:TEjest} and apply Proposition \ref{prop:TD1+key} to obtain 
\begin{equation}\label{equ:T12Fest}
\|T_{K_{k,j}}\|_{L^p(X)\to L^p(X)}\lesssim 
\begin{cases}
2^{-\frac{k}2}\quad\text{if}\quad p=2,\\
2^{-\frac2pk}\quad\text{if}\quad p>4,
\end{cases}
\end{equation}
for $j=1,2$. 
For the case with $j=3$, since the factor $e^{-i2^k(r_1+r_2)}$ does not affect $L^p$ norms, we have
\begin{equation}\label{equ:T3est}
\|T_{K_{k,3}}\|_{L^p(X)\to L^p(X)}\lesssim  \|T_{M_{k}}\|_{L^p(X)\to L^p(X)}+\|T_{R_{k}}\|_{L^p(X)\to L^p(X)},
\end{equation}
where $R_k(r_1,r_2; \theta)=e^{-i2^k(r_1+r_2)} K_{k,3}(r_1,r_2; \theta)-M_k(r_1,r_2; \theta)$. Due to \eqref{equ:TE3est}, we still have 
bounds \eqref{equ:T12Fest} for $\|T_{R_{k}}\|_{L^p(X)\to L^p(X)}$. So we are reduced to prove
\begin{equation}\label{equ:T3Mest}
 \|T_{M_{k}}\|_{L^p(X)\to L^p(X)} \lesssim 
\begin{cases}
2^{-\frac{k}2}\quad\text{if}\quad p=2,\\
2^{-\frac2pk}\quad\text{if}\quad p>4.
\end{cases}
\end{equation}

Let us introduce a even bump function $\chi\in C_c^\infty([-\epsilon,\epsilon])$ such that $\chi(\theta)=1$ for $|\theta|\leq\frac{\epsilon}2$ and 
$\chi(\theta)=0$ when $|\theta|\geq\epsilon$, where $0<\epsilon\ll1$.
If $|\pi\pm \theta|\geq\frac{\epsilon}2$, we obtain
$$\int_0^{2\sigma}\left|\frac{\pa}{\pa s}\left(\frac{4\sigma(\pi\pm\theta)}{s^2+(\pi\pm\theta)^2} \right)\right|\;ds\leq \frac{C}{|\pi\pm \theta|}\leq \tilde{C}.$$
So by Van der Corput Lemma, on the support of $1-\chi(\pi\pm\theta)$, we have
$$\left|\int_0^{2\sigma}e^{i2^k\frac{r_1r_2}{r_1+r_2}s^2}\frac{4\sigma(\pi\pm\theta)}{s^2+(\pi\pm\theta)^2}\;ds\right|\leq C\left(2^k\frac{r_1r_2}{r_1+r_2}\right)^{-\frac12}.$$
Hence,
$$\left|\big(1-\chi(\pi\pm\theta)\big)\chi_1(2^kr_1r_2)\psi(r_1+r_2)M_k(r_1,r_2,\theta)\right|\lesssim (1+2^{k}r_1r_2)^{-\frac12}.$$
This together with Proposition \ref{prop:TD1+key} shows \eqref{equ:T3Mest} with the cut off function $1-\chi(\pi\pm\theta)$.
Therefore, we are reduced to estimate the term
\begin{equation}
\begin{split}
T_{\chi M_k}f=\int_0^{\infty}&\int_0^{2\sigma\pi}\chi_1(2^kr_1r_2)\psi(r_1+r_2)\\
&\chi(\pi-(\theta_1-\theta_2))M_k(r_1,r_2,\theta_1-\theta_2)f(r_2,\theta_2)\;d\theta_2\;r_2dr_2.
\end{split}
\end{equation}

\begin{lemma}\label{lem:TchiN}
There holds
\begin{equation}\label{equ:T12N2est}
\|T_{\chi M_k}\|_{L^p(X)\to L^p(X)}\lesssim 
\begin{cases}
2^{-\frac{k}2}\quad&\text{if}\quad p=2,\\
2^{-\frac2pk}\quad&\text{if}\quad p>4.
\end{cases}
\end{equation}

\end{lemma}

\begin{proof}
We write
\begin{equation}\label{equ:Hr1r2def}
H_k(r_1,r_2;\theta):=\chi_1(2^kr_1r_2)\psi(r_1+r_2)\chi(\pi-\theta)M_k(r_1,r_2,\theta).
\end{equation}
We first consider the case  $p=2$.
By the same argument as \cite[Lemma 2.3]{BFM18}, we obtain
\begin{equation}\label{equ:Hr1est}
\begin{split}
\big|\mathcal{F}_{\theta\mapsto\zeta}\big(H_k(r_1,r_2;\cdot)\big)(\zeta)\big|&\lesssim
\begin{cases}
(2^kr_1r_2)^{-\frac12},\quad&\text{if}\quad |\zeta|\lesssim(2^kr_1r_2)^{\frac12}\\
\frac1{|\zeta|},\quad&\text{if}\quad |\zeta|\gg(2^kr_1r_2)^{\frac12}
\end{cases}\\
&\lesssim(2^kr_1r_2)^{-\frac12}.
\end{split}
\end{equation}
On the support of $\psi$ (i.e. $r_1+r_2\simeq 1$) and $\chi$ (i.e. $\theta_1-\theta_2\in [\pi-\epsilon,\pi+\epsilon]$), we use the Plancherel's theorem and \eqref{equ:Hr1est}
to obtain
\begin{align*}
&\|T_{\chi M_k}f\|_{L^2(X)}\\
=&\left\| \int_0^{\infty}\int_0^{2\sigma\pi}H_k(r_1,r_2,\theta_1-\theta_2)f(r_2,\theta_2)\;d\theta_2\;r_2dr_2\right\|_{L^2_{\theta_1}([0,2\sigma\pi],L^2(r_1dr_1))}\\
\leq&\left\| \int_0^1\left\|\int_{\R}H_k(r_1,r_2,\theta_1-\theta_2)\mathbbm{1}_{[0,2\sigma\pi]}(\theta_2)f(r_2,\theta_2)\;d\theta_2\right\|_{L^2(\R)}\;r_2dr_2\right\|_{L^2([0,1],r_1dr_1))}\\
\lesssim&\left\| \int_0^1\left\|\mathcal{F}_{\theta\mapsto\zeta}\big(H_k(r_1,r_2;\cdot)\big)(\zeta)\right\|_{L^\infty_\zeta(\R)}
\left\|\mathcal{F}_{\theta\mapsto\zeta}\big(\mathbbm{1}_{[0,2\sigma\pi]}f(r_2,\cdot)\big)\right\|_{L^2_\zeta(\R)}\;r_2dr_2\right\|_{L^2([0,1],r_1dr_1))}\\
\lesssim&\left\| \int_0^1(2^kr_1r_2)^{-\frac12}
\left\|\mathbbm{1}_{[0,2\sigma\pi]}f(r_2,\theta_2)\right\|_{L^2_\zeta(\R)}\;r_2dr_2\right\|_{L^2([0,1],r_1dr_1))}\\
\lesssim&2^{-\frac{k}2}\|f\|_{L^2(X)}.
\end{align*}
Next we consider the case $p>4$. By the same argument as \cite[Lemma 5.4]{MYZ}, we also have that
\begin{equation}\label{equ:Nsquarees}
\big|H_k(r_1,r_2;\theta)\big|\leq C(1+2^kr_1r_2(\pi-\theta)^2)^{-\frac12},
\end{equation}
which implies, for any $0<\epsilon_1\ll 1$,
\begin{align*}
\int_{\R}\big| H_k(r_1,r_2;\theta)\big|\;d\theta\lesssim&\int_{-\epsilon}^{\epsilon}(1+2^kr_1r_2\theta^2)^{-\frac12}\;d\theta\\
\lesssim&(2^kr_1r_2)^{-\frac12}\int_{|\theta|\leq \epsilon \sqrt{2^kr_1r_2}}\frac1{1+\theta}\;d\theta\\
\lesssim&(2^kr_1r_2)^{-\frac12+\epsilon_1}\lesssim(1+2^kr_1r_2)^{-\frac12+\epsilon_1}.
\end{align*}
By Young's inequality in $\theta$ and $r_1+r_2\simeq 1$, for $p>4$, we have
\begin{align*}
&\|T_{\chi M_k}f\|_{L^p(X)}\\
=&\left\| \int_0^{\infty}\int_0^{2\sigma\pi}H(r_1,r_2,\theta_1-\theta_2)f(r_2,\theta_2)\;d\theta_2\;r_2dr_2\right\|_{L^p_{\theta_1}([0,2\sigma\pi],L^p(r_1dr_1))}\\
\lesssim&\left\| \int_0^1 \|H(r_1,r_2;\cdot)\|_{L_\theta^1(\R)}\|f(r_2,\cdot)\|_{L_\theta^p([0,2\sigma\pi])}\;r_2dr_2\right\|_{L^p([0,1];r_1dr_1))}\\
\lesssim&\left\| \int_0^1(1+2^kr_1r_2)^{-\frac12+\epsilon_1}\|f(r_2,\cdot)\|_{L_\theta^p([0,2\sigma\pi])}\;r_2dr_2\right\|_{L^p([0,1];r_1dr_1))}\\
\lesssim& \int_0^1 \big\|(1+2^kr_1r_2)^{-\frac12+\epsilon_1}\big\|_{L^p([0,1];r_1dr_1))}\|f(r_2,\cdot)\|_{L_\theta^p([0,2\sigma\pi])}\;r_2dr_2\\
\lesssim&2^{-\frac2pk} \int_0^1  \big\|(1+r_1)^{-\frac12+\epsilon_1}\big\|_{L^p([0,+\infty);r_1dr_1))}r_2^{-\frac2p}\|f(r_2,\cdot)\|_{L_\theta^p([0,2\sigma\pi])}\;r_2dr_2\\
\lesssim&2^{-\frac2pk}\left( \int_0^1 r^{-\frac2pp'}r_2\;dr_2\right)^{\frac1{p'}}\|f\|_{L^p(X)}\\
\lesssim&2^{-\frac2pk}\|f\|_{L^p(X)},
\end{align*}
which completes the proof of Lemma \ref{lem:TchiN}.
\end{proof}

\begin{remark}
For $p=4$, similarly we have
\begin{align*}
&\|T_{\chi M_k}f\|_{L^4(X)}\\
\lesssim&\left\| \int_0^1 \|H(r_1,r_2;\cdot)\|_{L_\theta^1(\R)}\|f(r_2,\cdot)\|_{L_\theta^p([0,2\sigma\pi])}\;r_2dr_2\right\|_{L^p([0,1];r_1dr_1))}\\
\lesssim& \int_0^1 \Big(\int_0^1(1+2^kr_1r_2)^{-2+4\epsilon_1}r_1 dr_1\Big)^{\frac14}\|f(r_2,\cdot)\|_{L_\theta^4([0,2\sigma\pi])}\;r_2dr_2\\
\lesssim&2^{-k(\frac 12-\epsilon_1)} \int_0^1  r_2^{-\frac12}\|f(r_2,\cdot)\|_{L_\theta^4([0,2\sigma\pi])}\;r_2dr_2\\
\lesssim&2^{-k(\frac 12-\epsilon_1)}\left( \int_0^1 r^{-\frac23}r_2\;dr_2\right)^{\frac3{4}}\|f\|_{L^4(X)}\\
\lesssim&2^{-k(\frac 12-\epsilon_1)}\|f\|_{L^4(X)}.
\end{align*}
We see that there is a loss when $p=4$, which is the most `critical' case and this corresponds to the fact that the Bochner-Riesz mean operator is only bounded on $L^4$ for $\delta>\delta_c(p=4,n=2)=0$.
\end{remark}

By Lemma \ref{lem:TchiN} and the above argument, we obtain \eqref{equ:T3Mest}.  Combining this with \eqref{equ:widTK1fest}, we get \eqref{equ:TD1+kesti-re}. Therefore, we conclude the proof of  Proposition \ref{prop:keyTD1+k}.

\end{proof}

%\subsection{Proof of Proposition \ref{prop:TEjest}}
%
%In this subsection, we prove  Proposition \ref{prop:TEjest}. Recall
%\begin{align*}
%E_1=&\int_0^{2\sigma} e^{i2^kd_s}\frac{a_+(2^kd_s)}{(1+2^kd_s)^{\frac{3}2+\delta}}
%\left(\frac{\sin\big(\tfrac{\theta_1-\theta_2}{\sigma}\big)}{\cosh\frac{s}{\sigma}-\cos\big(\tfrac{(\theta_1-\theta_2}{\sigma}\big)}
%-\frac{4\sigma(\theta_1-\theta_2)}{s^2+(\theta_1-\theta_2)^2}\right)
%\;ds,\\
%E_2=&\int_0^{2\sigma} e^{i2^kd_s}\left(\frac{a_+(2^kd_s)}{(1+2^kd_s)^{\frac{3}2+\delta}}-\frac{a_+(2^k(r_1+r_2))}{(1+2^k(r_1+r_2))^{\frac32+\delta}}\right)\frac{4\sigma(\theta_1-\theta_2)}{s^2+(\theta_1-\theta_2)^2}\;ds.\\
%\end{align*}

%Thus, we can regard $N(r_1,r_2,\theta)$ as the precise approximation to $e^{-2^k(r_1+r_2)}E_3$.

\appendix
\section{Appendix}\label{Sec:appendix}

In this section, we recall basic lemmas (e.g. Van der Corput Lemma) about oscillatory integrals and prove the general Young's inequality on the flat cones $X$.

\subsection{Basic harmonic analysis on cones}
Recall that the flat cones  $X=C(\mathbb{S}_\sigma^1)=(0,\infty)\times\mathbb{S}_\sigma^1$ is a product cone over the circle $\mathbb{S}_\sigma^1=\R/2\pi\sigma\Z$ with radius $\sigma>0$ and the metric $g=dr^2+r^2d\theta^2$.  Let  $x=(r, \theta)\in
\R_+\times \Sa$, then the measure on $C(\Sa)$ is
\begin{equation}\label{equ:measX}
dx=rdr d\theta.
\end{equation}
For $1\leq p<\infty$,  we define the $L^p(X)$ space by the complement of $\mathcal{C}_0^\infty(X)$ under the norm
\begin{equation}\label{equ:Lpdef}
\|f\|_{L^p(X)}^p=\int_{X}|f(x)|^p dx=\int_0^\infty\int_{\Sa}  |f(r,\theta)|^p d\theta\;rdr
.\end{equation}
Let $d$ be the distance function on $X=C(\Sa)$,  then, for instance see \cite{CT}, the distance on a flat cone is
\begin{equation}\label{equ:distdef}
d(x,y)=\begin{cases}\sqrt{r_1^2+r_2^2-2r_1r_2\cos(\theta_1-\theta_2)},\quad &|\theta_1-\theta_2|\leq \pi;\\
r_1+r_2, &|\theta_1-\theta_2|\geq \pi,\
\end{cases}
\end{equation}
with $x=(r_1,\theta_1)$ and $y=(r_2,\theta_2)$ in $C(X)$.
%Furthermore, about the distance function, we refer the reader to Li \cite[Proposition 1.3, Lemma 3.1]{L2} for the following results.

\begin{lemma}\label{lem:logYoung}
Let $1\leq p\leq+\infty$. Then, we have
\begin{equation}\label{equ:basYoung}
\left(\int_0^\infty\left|\int_0^\infty \frac{f(r_2)}{(1+r_1+r_2)^3}r_2dr_2\right|^p\;r_1dr_1\right)^\frac1p\leq C\left(\int_0^\infty|f(r)|^pr\;dr\right)^\frac1p.
\end{equation}

\end{lemma}

\begin{proof}
By change of variables and Young's inequality, we obtain
\begin{align*}
&\left(\int_0^\infty\left|\int_0^\infty \frac{f(r_2)}{(1+r_1+r_2)^3}r_2dr_2\right|^p\;r_1dr_1\right)^\frac1p\\
\lesssim&\left(\int_0^\infty\left|\int_0^\infty \frac{f(\sqrt{r_2})}{(1+r_1+r_2)^\frac32}dr_2\right|^p\;dr_1\right)^\frac1p\\
\lesssim&\left(\int_0^\infty\left|\int_0^\infty \frac{f(\sqrt{r_2-r_1})\chi_{r_2-r_1\geq0}}{(1+r_2)^\frac32}dr_2\right|^p\;dr_1\right)^\frac1p\\
\lesssim&\left(\int_0^\infty|f(\sqrt{r})|^p\;dr\right)^\frac1p \int_0^\infty \frac1{(1+r)^\frac32}\;dr\\
\lesssim&\left(\int_0^\infty|f(r)|^pr\;dr\right)^\frac1p.
\end{align*}
And so \eqref{equ:basYoung} follows and we have completed the proof of Lemma \ref{lem:logYoung}.

\end{proof}

We also need to establish the following general Young's inequality.

\begin{proof}[The proof of Lemma \ref{lem:GYoung} ]
Define $F(r,\theta)$ by
\begin{equation*}
F(r,\theta)= \begin{cases}
f(r,\theta),\quad&\text{if}\;0\leq \theta\leq 2\sigma\pi,\\
0,\quad&\text{if}\;\theta> 2\sigma\pi.
\end{cases}
\end{equation*}
For any fixed $\theta_0\in\R,\sigma>0$,  there exists $k_1,k_2\in\Z$ such that
$$[-\theta_0,2\sigma\pi-\theta_0]\subset[2k_1\pi,2k_2\pi].$$
Replacing $\theta_1$ by $\theta_1+\theta_0$, we obtain
\begin{align*}
&\text{LHS of}\; \eqref{equ:GYoung}\\
\leq&\left\|\int_0^\infty \int_{0}^{2\sigma \pi}\frac{|F(r_2,\theta_2)|}{(1+\sqrt{r_1^2+r_2^2-2\cos(\theta_1-\theta_2)r_1r_2})^\alpha}\;d\theta_2r_2dr_2\right\|_{L^p_{\theta_1}([2k_1\pi,2k_2\pi], L^p(r_1dr_1))},
\end{align*}
which is further less than 
\begin{align*}
\leq&\sum_{k=0}^{[\sigma]}\left\|\int_0^\infty \int_{2k\pi}^{2(k+1)\pi}\frac{|F(r_2,\theta_2)|}{(1+\sqrt{r_1^2+r_2^2-2\cos(\theta_1-\theta_2)r_1r_2})^\alpha}\;d\theta_2r_2dr_2\right\|_{L^p_{\theta_1}([2k_1\pi,2k_2\pi], L^p(r_1dr_1))}\\
\leq&\sum_{k=0}^{[\sigma]}(k_2-k_1)^{1/p}\left\|\int_0^\infty \int_{0}^{2\pi}\frac{|F(r_2,\theta_2+2k\pi)|}{(1+\sqrt{r_1^2+r_2^2-2\cos(\theta_1-\theta_2)r_1r_2})^\alpha}\;d\theta_2r_2dr_2\right\|_{L^p_{\theta_1}([0,2\pi], L^p(r_1dr_1))}.
\end{align*}
In the last inequality, we use the periodicity of $\cos$-function. Let $F_k(r_2,\theta_2)=F(r_2,\theta_2+2k\pi)$, $x=r_1(\cos\theta_1,\sin\theta_1)$ and $y=r_2(\cos\theta_2,\sin\theta_2)$, then 
\begin{align*}
&\text{LHS of}\; \eqref{equ:GYoung}\\
\leq&\sum_{k=0}^{[\sigma]}(k_2-k_1)^{1/p}\left\|\int_0^\infty \int_{0}^{2\pi}\frac{|F_k(r_2,\theta_2)|}{(1+\big|r_1(\cos\theta_1,\sin\theta_1)-r_2(\cos\theta_2,\sin\theta_2)\big|)^\alpha}\;d\theta_2r_2dr_2\right\|_{L^p_{\theta_1}([0,2\pi], L^p(r_1dr_1))}\\
\leq&\sum_{k=0}^{[\sigma]}(k_2-k_1)^{1/p}\left\|\int_{\R^2}\frac{|F_k(y)|}{(1+|x-y|)^\alpha}\;dy\right\|_{L^p_x(\R^2)}.
\end{align*}
Hence, by Young's inequality and by definition of $F(r,\theta)$, we get
\begin{align*}
\text{LHS of}\; \eqref{equ:GYoung}\leq&C(\sigma,k_1,k_2)\sup_{0\leq k\leq[\sigma]}\|F_k(y)\|_{L^p(\R^2)}\\
\leq&C(\sigma,k_1,k_2)\sup_{0\leq k\leq[\sigma]}\left(\int_0^\infty \int_0^{2\pi} \big|F(r,\theta+2k\pi)\big|^p\;d\theta\;rdr\right)^\frac1p\\
\leq&C(\sigma,k_1,k_2)\left(\int_0^\infty \int_0^{2\sigma\pi} \big|f(r,\theta)\big|^p\;d\theta\;rdr\right)^\frac1p\\
\leq&C\|f\|_{L^p(X)}.
\end{align*}

\end{proof}

\begin{proposition}\label{prop:TD1+key}
Let $k\geq1$. Define 
$$T_Kf(r_1,\theta_1):=\int_0^\infty\int_0^{2\sigma\pi}K(k;r_1,r_2,\theta_1,\theta_2)f(r_2,\theta_2)\;d\theta_2\;r_2dr_2,$$
and the kernel $K(k;r_1,r_2,\theta_1,\theta_2)$ satisfies
\begin{equation}\label{equ:KTD1prop}
\big|K(k;r_1,r_2,\theta_1,\theta_2)\big|\leq C
(1+2^kr_1r_2)^{-\frac12}\psi(r_1+r_2),
\end{equation}
where $\psi\in C_c^\infty\big(\big[\tfrac34,\tfrac83\big]\big).$
Then,  we have
\begin{equation}\label{equ:TKfest}
\|T_K\|_{L^p(X)\to L^p(X)}\lesssim 
\begin{cases}
2^{-\frac2pk}\quad\text{if}\quad p>4,\\
2^{-\frac{1+\epsilon}2k}\quad\text{if}\quad p=4,\\
2^{-\frac{k}2}\quad\text{if}\quad \frac43<p<4,
\end{cases}
\end{equation}
with $0<\epsilon\ll1$.

\end{proposition}

\begin{proof}
For simplicity, we identify $f(r_2,\theta_2)$ with $f(y)=f(r_2\cos\theta_2,r_2\sin\theta_2)$. By \eqref{equ:KTD1prop} and
H\"older's inequality, we get for $\frac43<p<4$
\begin{align*}
&\|T_Kf\|_{L^p(X)}\\
\lesssim&\left\|
\int_0^\infty\int_0^{2\sigma\pi}(1+2^kr_1r_2)^{-\frac12}\psi(r_1+r_2)\big|f(r_2,\theta_2)\big|\;d\theta_2\;r_2dr_2\right\|_{
L^p_{\theta_1}([0,2\sigma\pi],L^p(r_1dr_1))}\\
\lesssim&2^{-\frac{k}2}
\left\| r_1^{-\frac12}\int_0^2\int_0^{2\sigma\pi}r_2^{-\frac12}\big|f(r_2,\theta_2)\big|\;d\theta_2\;r_2dr_2\right\|_{
L^p_{\theta_1}([0,2\sigma\pi],L^p([0,2],r_1dr_1))}\\
\lesssim&2^{-\frac{k}2}\| r_1^{-\frac12}\|_{L^p([0,2],r_1dr_1)}\|r_2^{-\frac12}\|_{L^{p'}([0,2],r_2dr_2)}\|f\|_{L^p(X)}\\
\lesssim&2^{-\frac{k}2}\|f\|_{L^p(X)},
\end{align*}
where we need the restriction $\frac43<p<4$ to guarantee that the above integrals converge.
While for $p>4$, by \eqref{equ:KTD1prop}, we obtain
\begin{align*}
&\|T_Kf\|_{L^p(X)}\\
\lesssim&\left\|
\int_0^\infty\int_0^{2\sigma\pi}(1+2^kr_1r_2)^{-\frac12}\psi(r_1+r_2)\big|f(r_2,\theta_2)\big|\;d\theta_2\;r_2dr_2\right\|_{
L^p_{\theta_1}([0,2\sigma\pi],L^p(r_1dr_1))}\\
\lesssim&\int_0^2\int_0^{2\sigma\pi}\|(1+2^kr_1r_2)^{-\frac12}\|_{L^p([0,2],r_1dr_1)}\big|f(r_2,\theta_2)\big|\;d\theta_2\;r_2dr_2\\
\lesssim&2^{-\frac2pk}\|(1+r_1)^{-\frac12}\|_{L^p([0,\infty),r_1dr_1)}\int_0^2\int_0^{2\sigma\pi}r_2^{-\frac2p}\big|f(r_2,\theta_2)\big|\;d\theta_2\;r_2dr_2\\
\lesssim&2^{-\frac2pk}\Big(\int_0^2r_2^{-\frac2pp'}r_2\;dr_2\Big)^\frac1{p'}\|f\|_{L^p(X)}\\
\lesssim&2^{-\frac2pk}\|f\|_{L^p(X)},
\end{align*}
where we need the restriction $p>4$ to ensure that $(1+r)^{-\frac12}\in L^p([0,\infty),rdr)$.

By interpolation, we obtain the case $p=4$, which completes the proof of Proposition \ref{prop:TD1+key}.

\end{proof}

\subsection{Some basic lemmas}
In this subsection, we recall two basic lemmas about the oscillatory integrals, see Stein \cite{Stein}.

\begin{lemma}\label{lem:nonstapha}
Let $\phi$ and $\psi$ be smooth functions so that $\psi$ has compact support in $(a,b)$, and $\phi'(x)\neq 0$ for all $x\in[a,b]$. Then, we have
\begin{equation}\label{equ:lamdec}
  \big|\int_a^b e^{i\lambda \phi(x)}\psi(x)\;dx\big|\leq C_K (1+\lambda)^{-K},\quad \forall\;K\geq0.
\end{equation}
\end{lemma}

\begin{lemma}[Van der Corput ] \label{lem:VCL} Let $\phi$ be real-valued and smooth in $(a,b)$, and that $|\phi^{(k)}(x)|\geq1$ for all $x\in (a,b)$. Then
\begin{equation}
\left|\int_a^b e^{i\lambda\phi(x)}\psi(x)dx\right|\leq c_k\lambda^{-1/k}\left(|\psi(b)|+\int_a^b|\psi'(x)|dx\right)
\end{equation}
holds when (i) $k\geq2$ or (ii) $k=1$ and $\phi'(x)$ is monotonic. Here $c_k$ is a constant depending only on $k$.
\end{lemma}

\begin{center}

\end{center}

\end{document}